\documentclass[journal,twoside,web]{ieeecolor}
\usepackage{mymacros}
\usepackage{TAC}

\def\BibTeX{{\rm B\kern-.05em{\sc i\kern-.025em b}\kern-.08em
    T\kern-.1667em\lower.7ex\hbox{E}\kern-.125emX}}
\begin{document}
\title{Automatic Performance Estimation for Decentralized Optimization}
\author{Sébastien Colla and Julien M. Hendrickx
\thanks{S. Colla and J. M. Hendrickx are  with  the  ICTEAM institute, UCLouvain, Louvain-la-Neuve, Belgium. S. Colla is supported by the French Community of Belgium through a FRIA fellowship (F.R.S.-FNRS). J. M. Hendrickx is supported by the “RevealFlight” Concerted Research Action (ARC) of the Federation Wallonie-Bruxelles and by the Incentive Grant for Scientific Research (MIS) “Learning from Pairwise Comparisons” of the F.R.S.-FNRS. Email addresses: {\tt\small sebastien.colla@uclouvain.be}, {\tt\small julien.hendrickx@uclouvain.be}}}

\maketitle

\begin{abstract}
We present a methodology to automatically compute worst-case performance bounds for a large class of first-order decentralized optimization algorithms. These algorithms aim at minimizing the average of local functions that are distributed across a network of agents. They typically combine local computations and consensus steps.
Our methodology is based on the approach of Performance Estimation Problem (PEP), which allows computing the worst-case performance and a worst-case instance of first-order optimization algorithms by solving an SDP.
We propose two ways of representing consensus steps in PEPs, which allow writing and solving PEPs for decentralized optimization.
The first formulation is exact but specific to a given averaging matrix.
The second formulation is a relaxation but provides guarantees valid over an entire class of averaging matrices, characterized by their spectral range. This formulation often allows recovering a posteriori the worst possible averaging matrix for the given algorithm.
We apply our methodology to three different decentralized methods. For each of them, we obtain numerically tight worst-case performance bounds that significantly improve on the existing ones, as well as insights about the parameters tuning and the worst communication networks.
\end{abstract}


\section{Introduction}
\label{sec:introduction}
\IEEEPARstart{T}{he} goal of this paper is to develop a methodology that automatically provides numerically tight performance bounds for first-order decentralized methods on convex functions and to demonstrate its usefulness on different existing methods.
Decentralized optimization has received an increasing attention due to its useful applications in large-scale machine learning and sensor networks, see e.g. \cite{DGD} for a survey. In decentralized methods for separable objective functions, we consider a set of agents $\{1,\dots,N\}$, working together to solve the following optimization problem: \vspace{-2mm}
\begin{equation} \label{opt:dec_prob}
  \underset{\text{\normalsize $x \in \Rvec{d}$}}{\mathrm{minimize}} \quad f(x) = \frac{1}{N}\sum_{i=1}^N f_i(x), \vspace{-1mm}
\end{equation}
where $f_i: \Rvec{d}\to\mathbb{R}$ is the private function locally held by agent $i$.
To achieve this goal, each agent $i$ holds its own version $x_i$ of the decision variable $x \in \Rvec{d}$. Agents perform local computations and exchange local information with their neighbors to come to an agreement on the minimizer $x^*$ of the global function $f$. Exchanges of information often take the form of an average consensus step on some quantity, e.g., on the $x_i$. This corresponds to a multiplication by an averaging matrix $W \in \Rmat{N}{N}$, typically assumed symmetric and \emph{doubly stochastic}, i.e., a nonnegative matrix whose rows and columns sum to one. This matrix $W$ indicates both the topology of the network of agents and the weights they are using during the average consensus. Therefore, we call it network or averaging matrix, without distinction.

One of the simplest decentralized optimization method is the distributed (sub)gradient descent (DGD) \cite{DsubGD} where agents successively perform an average consensus step \eqref{eq:DGD_cons} and a local gradient step \eqref{eq:DGD_comp}:
\vspace*{-4mm}
\begin{align}
  y_i^k &= \sum_{j=1}^N w_{ij} x_j^k,          \hspace*{15mm} \label{eq:DGD_cons} \\[-0.5mm]
  x_i^{k+1} &= y_i^k - \alpha^k \nabla f_i(x_i^k), \hspace*{15mm} \label{eq:DGD_comp} \vspace{-2mm}
\end{align}
for step-sizes $\alpha^k > 0$.
Numerous other methods rely on the interplay of gradient and average consensus steps,
such as EXTRA \cite{EXTRA}, DIGing \cite{DIGing}, and NIDS \cite{NIDS}. The convergence results of DIGing holds for time-varying network matrices and the algorithm can also be extended to handle directed communications \cite{DIGing}. Different other algorithms can handle directed communications \cite{directed,directed2}.
There also are accelerated decentralized methods, such as the accelerated distributed Nesterov gradient descent (Acc-DNGD) \cite{AccDNGD}. \\
Furthermore, average consensus is used in different distributed primal dual methods such as DDA \cite{DDA}, MSDA \cite{MSDA} or MSPD \cite{MSPD}, though not directly in all, e.g., D-ADMM \cite{ADMM_1,ADMM}.

We note in particular recent results have reached optimal performance for certain specific performance criteria and classes of functions \cite{optimal_gradient_tracking}, \cite{MSPD}.  However, many questions and challenges remain open in the design and analysis of decentralized optimization method.
For further information on the topic, we refer the reader to the introduction of this recent thesis \cite{H_Hendrikx_thesis} in the field.

The quality of an optimization method is often evaluated \textit{via} a worst-case guarantee. Having accurate performance bounds for a decentralized algorithm is important to correctly understand the impact of its parameters and the network topology on its performance and to compare it fairly with other algorithms. However, obtaining these guarantees using theoretical proofs can often be a challenging task, requiring combining the impact of the optimization component and the interconnection network, sometimes resulting in bounds that are conservative or overly complex.
For example, we will show in Section \ref{sec:NumRes} that the existing performance bounds for DGD or DIGing are significantly worse than the actual worst-cases.

\subsection{Contributions}
We propose a new analysis methodology allowing to compute numerically tight worst-case performance bounds of decentralized optimization methods. This methodology is based on an alternative computational approach that finds a worst-case performance guarantee of an algorithm by solving an optimization problem, known as the performance estimation problem (PEP), see Section \ref{sec:PEPcen} for more details.
The PEP approach has led to many results in centralized optimization, see e.g. \cite{PEP_Smooth,PEP_composite}, but it has never been exploited in decentralized optimization.
The current PEP framework lacks ways of representing the communications between the agents. Therefore, we propose two formulations of the average consensus steps that can be embedded in a solvable PEP. Both formulations are presented in Section \ref{sec:consensusPEP}.
The first one uses a given averaging matrix $W$ and is exact. It can be used in PEP with any matrix $W$ and leads to performance bounds that are tight, but specific to the given matrix. However, these bounds are too specific to understand the general behavior of the algorithm. Our second formulation thus considers entire spectral classes of symmetric matrices, as often found in the literature.
This allows the PEP problem to obtain spectral upper bounds on the performance that are valid over an entire spectral class of network matrices and to look for the worst matrix in the given class. Although this formulation is a relaxation, we observe tight results in most cases, see Section \ref{sec:NumRes}.
This spectral formulation is our main methodological contribution.

Using these two new formulations, the PEP approach can be applied directly to a large class of first-order decentralized algorithms, in a wide range of settings and problems. We demonstrate these new formulations by analyzing the worst-case performance of DGD \cite{DGD}, DIGing \cite{DIGing} and Acc-DNGD \cite{AccDNGD} in Section \ref{sec:NumRes}. For all three algorithms, the spectral formulation leads to tight spectral performance bounds and we observe that these bounds are actually independent of the number of agents $N$ (when $N \ge 2$).
For DGD and DIGing, these spectral bounds significantly improve on the existing theoretical ones. For DIGing, we also show how to use the PEP approach to obtain linear convergence rate guarantees on an infinite horizon. For Acc-DNGD, we use our tool to analyze the impact of the parameters and nuance the conjecture from \cite{AccDNGD} about the asymptotic convergence rate $\bigO(\frac{1}{K^2})$, where $K$ is the total number of iterations.

One of the advantages of this new methodology is its great flexibility, allowing to answer many different questions, for example, by changing the class of functions or the performance criterion. In future works, this tool can easily be used to analyze the effect of inexact communications or gradient computations. This methodology could also consider performance criteria that not only account for efficiency but also robustness against errors in gradient computations or communications in order to help designing algorithms with the best efficiency-robustness trade-off.

A preliminary version of these results was published in \cite{PEP_dec}.
Our main new contributions with respect to \cite{PEP_dec} are (i) the proof of the necessary constraints used in our spectral formulation for any dimension $d$ of the variables $x_i \in \Rvec{d}$; (ii) the technique to find the worst averaging matrix based on the worst-case solutions of the spectral formulation; (iii) a wider demonstration of the methodology by analyzing DIGing and Acc-DNGD algorithms in addition of DGD.

\subsection{Related work}

An alternative approach with similar motivations for automated performance evaluation and inspired by dynamical systems concepts is proposed in \cite{IQC}.  Integral quadratic constraints (IQC), usually used to obtain stability guarantees on complex dynamical systems, are adapted to obtain sufficient conditions for the convergence of optimization algorithms. It provides infinite-horizon linear rates of convergence, based on relatively small size problems, but it only applies when the convergence is geometric. In comparison, the PEP approach computes the worst-case performance on a finite horizon, and therefore, the size of the problem grows with the limit of the horizon. However, this allows PEP to analyze non-geometric convergence and the impact of time-varying properties.
Unlike PEP, the IQC approach offers no a priori guarantee of tightness, though it turns out to be tight in certain situations.

An application of the IQC methodology to decentralized optimization is presented in \cite{IQC_dec} and is also exploited for designing a new decentralized algorithm that achieves a faster worst-case linear convergence rate in the smooth strongly convex case. This IQC formulation uses problems whose size is independent of the number of agents $N$ and it considers a fixed framework of decentralized algorithms that embeds a lot of explicit first-order methods such as DIGing \cite{DIGing}, EXTRA \cite{EXTRA} and NIDS \cite{NIDS}.
This methodology cannot be directly applied to DGD, nor to smooth convex functions or any other situation that does not have a geometric convergence.
Our PEP approach applies directly to any decentralized method (implicit or explicit) that involves consensus steps in the form of a matrix product, without the need to cast the method into a specific framework. This makes PEP very modular and user-friendly, in particular with the PESTO toolbox \cite{PESTO}.

\section{General PEP approach} \label{sec:PEPcen}
In principle, a tight performance bound on an algorithm could be obtained by running it on every single instance - function and initial condition - allowed by the setting considered and selecting the worst performance obtained. This would also directly provide an example of “worst” instance if it exists.
The performance estimation problem (PEP) formulates this abstract idea as a real optimization problem that maximizes the error measure of the algorithm result, over all possible functions and initial conditions allowed \cite{PEP_Drori}.
This optimization problem is inherently infinite-dimensional, as it contains a continuous function among its variables. Nevertheless, Taylor et al. have shown \cite{PEP_Smooth,PEP_composite} that PEP can be solved exactly using an SDP formulation, for a wide class of centralized first-order algorithms and different classes of functions.

The main ingredients of a centralized PEP are: (i) a performance measure $\mc{P}$, e.g. $f(x^k)-f(x^*)$; (ii) a class of functions $\mc{F}$, e.g. the class $\mc{F}_{\mu,L}$ of $\mu$-strongly convex and $L$-smooth functions; (iii) an optimization method $\mc{M}$ on class $\mc{F}$; (iv) a set $\mc{I}^0$ of initial conditions, e.g. $\|x^0-x^*\|^2 \le 1$.
These ingredients are organized into an optimization problem as \vspace{-1mm}
\begin{align}
  \underset{f, x^0,\dots, x^K, x^*}{\sup} \quad & \mathrm{\mc{P}}\qty(f, x^0,\dots,x^K, x^*) \hspace*{-4mm}  \label{eq:gen_PEP}\\[-0.3mm]
 \text{s.t.} \hspace{8mm}       & \hspace{-5mm}    f \in \mc{F}, \hspace{8mm}
                                 x^* \in \mathrm{argmin }~ f, \\[-0.6mm]
                                & \hspace{-5mm} x^k \text{ are iterates from method $\mc{M}$ applied on $f$,} \\[-0.6mm]
                                & \hspace{-5mm} \mc{I}^0 ~\text{ holds.} \\[-6mm]
\end{align}
To overcome the infinite dimension of variable $f$, the problem can be discretized into $\{\qty(x^k,g^k,f^k)\}_{k\in I}$, where $g^k$ and $f^k$ are respectively the (sub)gradient and the function value of $f$ at point $k$ and $I = \{0,\dots,K,*\}$. Then, the constraint $f \in \mc{F}$ is replaced by interpolation conditions ensuring that there exists a function of class $\mc{F}$ which interpolates those data points $\{\qty(x^k,g^k,f^k)\}_{k \in I}$, i.e. these values are consistent with an actual function.
Such constraints are provided for many different classes of functions in \cite[Section 3]{PEP_composite}. They are generally quadratics and potentially non-convex in the iterates and the gradients vectors, but they are linear in the scalar products of these and in the function values. The same holds true for most classical performance criteria and initial conditions.
We can then consider these scalar products directly as decision variables of the PEP. For this purpose, we define a Gram matrix $G$ that contains scalar products between all vectors, e.g. the iterates $x^k \in \Rvec{d}$ and subgradients $g^k \in \Rvec{d}$. \vspace{-0.25mm}
\begin{align}
  G = P^TP, \text{ with } P = \qty[g^0\dots g^K g^* x^0\dots x^K x^*]. \\[-5.5mm]
\end{align}
By definition, $G$ is symmetric and positive semidefinite.
Moreover, to every matrix $G \succeq 0$ corresponds a matrix $P$ whose has a number of rows equal to $\text{rank}~G$. We can show (see \cite{PEP_composite}) that the reformulation of \eqref{eq:gen_PEP} with $G$ is lossless provided that we do not impose the dimension $d$ in \eqref{eq:gen_PEP}, and indeed look for the worst-case over all possible dimensions (imposing the dimension would correspond to adding a typically less tractable rank constraint on $G$).
The idea is therefore to formulate a PEP of the form of \eqref{eq:gen_PEP} as an equivalent positive semidefinite program (SDP) using the Gram matrix $G$ and the vector of functional values $f_v = [f^k]_{k \in I}$ as variables.

This SDP formulation is convenient because it can be solved numerically to global optimality, leading to a tight worst-case bound, and it also provides the worst-case solution over all possible problem dimensions.
We refer the reader to \cite{PEP_composite} for more details about the SDP formulation of PEP, including ways of reducing the size of matrix $G$. However, the dimension of $G$ always depends on the number of iterations $K$.
From a solution $G$, $f_v$ of the SDP formulation, we can construct a solution for the discretized variables $\{\qty(x^k,g^k,f^k)\}_{k\in I}$, e.g. using the Cholesky decomposition of $G$. Since these points satisfy sufficient interpolation constraints, we can also construct a function from $\mc{F}$ interpolating these points.

Proposition \ref{prop:GramPEP} states sufficient conditions under which a PEP in the form of \eqref{eq:gen_PEP} can be formulated as an SDP. These conditions are satisfied in most PEP settings, allowing to formulate and solve the SDP PEP formulation for a large class of (implicit and explicit) first-order methods, with different classes of functions and performance criteria, see \cite{PEP_composite}. The proposition uses the following definition.
\begin{definition}[Gram-representable \cite{PEP_composite}] \label{def:Gram}
Consider a Gram matrix $G$ and a vector $f_v$, as defined above. We say that
a constraint or an objective is linearly (resp. LMI) Gram-representable if it can be expressed using a finite set of linear (resp. LMI) constraints involving (part of) $G$ and $f_v$.
\end{definition} \smallskip
\begin{proposition}[\hspace{-0.5pt}{{\cite[Proposition 2.6]{PEP_composite}}}] \label{prop:GramPEP}
  If the interpolation constraints of the class of functions $\mc{F}$, the satisfaction of the method $\mc{M}$, the performance measure $\mc{P}$ and the set of constraints $\mc{I}$, which includes the initial conditions, are linearly (or LMI) Gram-representable, then, computing the worst-case for criterion $\mc{P}$ of method $\mc{M}$ after $K$ iterations on objective functions in class $\mc{F}$ with constraints $\mc{I}$ can be formulated as an SDP, with $G \succeq 0$ and $f_v$ as variables. \\
  This remains valid when the objective function is the sum of $N$ sub-functions being each in a class of functions with linearly (or LMI) Gram representable interpolation constraints.
\end{proposition} \smallskip
\emph{Remark:}
In \cite{PEP_composite}, Definition \ref{def:Gram} and Proposition \ref{prop:GramPEP} were only formulated for linearly Gram-representable constraints, but their extension to LMI Gram-representable constraints is direct. Such constraints appear in the analysis of consensus steps with spectral classes of network matrices.

PEP techniques allowed answering several important questions in optimization, see e.g. the list in \cite{Taylor_thesis}, and to make important progress in the tuning of certain algorithms including the well-known centralized gradient-descent.
It was further exploited to design optimal first-order methods: OGM for smooth convex optimization \cite{OGM} and its extension ITEM for smooth strongly convex optimization \cite{ITEM}.

It can also be used to deduce proofs about the performance of the algorithms \cite{PEP_compo}. It has been made widely accessible \textit{via} a Matlab \cite{PESTO} and Python \cite{pepit2022} toolbox.
However, PEPs have never been used to study the performance of decentralized methods.

\section{Representing consensus steps in PEP} \label{sec:consensusPEP}
The main missing block to develop a PEP formulation for decentralized methods is to find a proper way of representing the consensus steps of the methods in the SDP PEP.
In PEP for decentralized methods, we must find the worst-case for each of the $N$ local functions $f_i$ and the $N$ sequences of local iterates $x_i^0 \dots x_i^K$ ($i=1,\dots,N$). The same techniques as in Section \ref{sec:PEPcen} can be applied to discretize the problem, using proper interpolation conditions on each local function. We also want to use a Gram matrix of scalar products to reformulate it as an SDP that can be solved efficiently. Therefore, according to Proposition \ref{prop:GramPEP}, we need to use linearly Gram representable performance criteria, classes of functions and initial conditions. Moreover, we also need to find a linearly or LMI Gram-representable way of expressing the updates of the decentralized method to be analyzed.
There is currently no representation of the consensus steps that can be embedded in the SDP formulation. Therefore, the rest of this section proposes ways to represent agent interactions in the SDP PEP formulation and thus provides the missing block for analyzing decentralized first-order optimization methods with PEP.

Agent interactions are part of any decentralized methods and often take place \textit{via} a weighted averaging, which can be described as a consensus step of the following form, similar to that used in DGD \eqref{eq:DGD_cons}:  \vspace{-1.5mm}
\begin{equation} \label{eq:consensus}
    y_i = \sum_{j=1}^N w_{ij} x_j, \qquad \text{for all $i\in\{1,\dots,N\}$,} \vspace{-1.5mm}
\end{equation}
where $x_j$ can represent any vector in $\Rvec{d}$ held by agent $j$, e.g., its local iterates in the case of DGD or something else in more advanced methods. Vector $y_i \in \Rvec{d}$ is an auxiliary variable that represents the result of the interaction and $W \in \Rmat{N}{N}$ is the averaging matrix.
This form of communication is used in many decentralized methods such as DGD \cite{DsubGD}, DIGing \cite{DIGing}, EXTRA \cite{EXTRA}, NIDS \cite{NIDS} and the results presented in this section can be exploited for all these methods.

When $K$ different consensus steps are involved in the algorithm, we observe that the Gram matrix $G$ of all scalar products contains the submatrix $G_c$:
\begin{align}
    G_c &= P_c^TP_c, \label{eq:Gc} \\[1mm]
    \small\text{  with  } P_c &= \begin{bmatrix} x_0^0 \dots x_N^0 \dots x_0^K \dots x_N^K~y_0^0 \dots y_N^0 \dots y_0^K \dots y_N^K \end{bmatrix}. \\[-8mm]
\end{align}
We will see that the new constraints we propose to represent the consensus steps are linearly or LMI Gram-representable in $G_c$, and then also linearly or LMI Gram-representable in $G$.

\subsection{Averaging matrix given a priori}
When the averaging matrix is given \emph{a priori}, the consensus step \eqref{eq:consensus} corresponds to a set of linear equality constraints on $y_i$ and $x_j$, which are equivalent to the constraints \vspace{-1mm}
$$\Big(y_i - \sum_{j=1}^N w_{ij} x_j\Big)^T\Big(y_i - \sum_{j=1}^N w_{ij} x_j\Big) = 0, \text{ for $i=1,\dots,N$.} \vspace{-1mm}$$
These constraints only involve scalar products from $G_c$ \eqref{eq:Gc} and are thus linearly Gram-representable. They can therefore be used in the SDP formulation of a PEP, see Proposition \ref{prop:GramPEP}. Alternatively, linear equality constraints \eqref{eq:consensus} can also be used to substitute the values of $y_i$ in the problem, allowing to reduce its size. This solution is also preferable for numerical reasons.
In any case, this allows writing PEPs that provide exact worst-case performances for the given decentralized method and the specific averaging matrix given \emph{a priori}. We call this the \textit{exact PEP formulation}. It can be applied to any matrix $W$, and not only for symmetric or doubly stochastic ones.
This can be useful for trying different network matrices and observing their impact on the worst-case performance of the algorithm. It will serve as an exact comparison baseline in the numerical experiments of Section \ref{sec:NumRes}.
The next section presents a way of representing communications in PEP that allows obtaining more general performance guarantees, valid over entire classes of network matrices and not only for a specific one.

\subsection{Averaging matrix as a variable}
We now consider that the matrix $W$ is not given \emph{a priori}, but is \emph{one of the decision variables} of the performance estimation problem with bounds on its possible eigenvalues. Hence, the PEP also looks for the worst averaging matrix among all the possible ones. Typically, the literature considers matrices that are symmetric,  doubly-stochastic and whose all eigenvalues (except $1$) are in a given range. We were unable to represent this class of matrices in the SDP PEP formulation, in particular the non-negativity assumption embedded in the doubly stochasticity. Therefore, we will consider a slightly more general class of matrices, where we relax this non-negativity assumption.

\begin{definition}[Generalized doubly stochastic matrix \cite{gds}] \label{def:gds} A matrix $W \in \Rmat{N}{N}$ is \emph{generalized doubly stochastic} if its rows and columns sum to one, i.e., if \vspace{-1.5mm}
  $$\sum_{i=1}^N w_{ij} = 1, \qquad \sum_{i=1}^N w_{ji} = 1, \quad \text{ for $j=1,\dots,N$.} \vspace{1.5mm}$$
\end{definition}
The resulting worst matrix of the PEP may thus have negative elements. However, the provided worst-case guarantees are also valid for non-negative matrices. Moreover, we note that most results from the literature exploiting spectral information of stochastic matrices do in fact not use the non-negativity of $W$ and are thus really about generalized stochastic matrices, see e.g. \cite{DGD, EXTRA, NIDS}. We analyze the impact of this relaxation in the case of DGD in Section \ref{sec:analysis_DGD}.

Formally, the search space for $W$ is restricted by the following constraints, for each agent $i\in\{1,\dots,N\}$ and each consensus steps $k\in\{1,\dots,K\}$, \vspace{-1mm}
\begin{align}
  y_i^k &= \sum_{j=1}^N w_{ij} x_j^k, \label{eq:cons_2} \\
  W &\in \Wcl{\lm}{\lp}, \label{eq:cons_1}\\[-5.5mm]
\end{align}
where $\Wcl{\lm}{\lp}$ is the set of real, symmetric and generalized doubly stochastic $N\times N$ matrices that have their eigenvalues between $\lm$ and $\lp$, except for $\lam_1 = 1$:
$$ \lm \le \lam_N\le\cdots\le\lam_2\le\lp \quad \text{where $\lm, \lp \in \qty(-1,1)$ }. $$

We do not have a direct way for representing constraints \eqref{eq:cons_2} and \eqref{eq:cons_1} in an LMI Gram-representable manner that can be embedded in an SDP PEP, hence we will use a relaxation of these constraints that we will see in Section \ref{sec:NumRes} is often close to tight. From constraints \eqref{eq:cons_2} and \eqref{eq:cons_1}, we derive new necessary conditions involving only variables $y_i^k$ and $x_i^k$, allowing to eliminate $W$ from the problem. We also show that these new constraints can be expressed in terms of $G_c$ \eqref{eq:Gc} and are therefore Gram-representable.

\subsubsection*{Notations} Let $x_i^k, y_i^k \in \Rvec{d}$ be the local variables of agent $i$ at iteration $k$ of an algorithm.
Let $X^j$, $Y^j \in \Rmat{N}{K}$ be matrices containing the $j^{\mathrm{th}}$ component of each of these local variables:
$$ X^j_{i,k} = (x_i^k)_j \quad \text{and}\quad Y^j_{i,k} = (y_i^k)_j  \qquad \text{for ~~$\substack{j=1,\dots,d \\ i=1,\dots,N \\ k=1,\dots,K}$ }$$
Each column corresponds thus to a different consensus step $k$, and each row to a different agent $i$. Using this notation, the consensus steps constraints \eqref{eq:cons_2} can simply be written as $d$ matrix equations: $Y^j = W X^j$ for each $j=1,\dots,d$.

Moreover, we can stack each $X^j$ and $Y^j$ vertically in matrices $X,Y \in \Rmat{Nd}{K}$, and write all the consensus steps \eqref{eq:cons_2} with one matrix equation:
$$\underbrace{\begin{bmatrix}  Y^{1} \\ \vdots \\ Y^{d} \end{bmatrix}}_{Y} = \underbrace{\begin{bmatrix}  W &&0 \\& \ddots &\\ 0&&W \end{bmatrix}}_{\Wt} \underbrace{\begin{bmatrix}  X^{1} \\ \vdots \\ X^{d} \end{bmatrix}}_{X}. $$
The matrix $\Wt \in \Rmat{Nd}{Nd}$ is block-diagonal repeating $d$ times $W$ and can be written as $\Wt = (I_d \otimes W)$, where $I_d \in \Rmat{d}{d}$ is the identity matrix and $\otimes$ denotes the Kronecker product. We decompose matrices $X$ and $Y$ in average and centered parts: \vspace{-1mm}
$$ X = (\Xb \otimes \mathbf{1}_N ) + \Xc, \qquad Y =  (\Yb \otimes \mathbf{1}_N) + \Yc,$$
where $\Xb$ and $\Yb$ are agents average vectors in $\Rmat{d}{K}$, defined as $\Xb_{\cdot k} = \frac{1}{N} \sum_{i=1}^N x_i^k$, $\Yb_{\cdot k} = \frac{1}{N} \sum_{i=1}^N y_i^k$ for $k=1,\dots,K$ and $\mathbf{1}_N = \qty[1\dots 1]^T \in \Rvec{N}$.
The centered matrices $\Xc, \Yc \in \Rmat{Nd}{K}$ have by definition an agent average of zero for each component and iteration: $(I_d \otimes \mathbf{1}_N)^T \Xc = \mathbf{0}_{d\times K}$.

\subsubsection*{Gram-representable relaxation of \eqref{eq:cons_2} and \eqref{eq:cons_1}}
\begin{theorem}[Consensus Constraints] \label{thm:conscons}
If $Y^j = WX^j$, for every $j=1,\dots,d$ and for a same matrix $W\in \Wcl{\lm}{\lp}$,
i.e., if $Y = (I_d \otimes W) X$, then
\begin{enumerate}[(i)]
  \item The matrices $X^TY$ and $\Xc^T\Yc$ are symmetric,
  \item The following constraints are satisfied \vspace{-0.5mm}
  \begin{align}
    \Xb &= \Yb, \label{eq:eq_mean} \\
    \lm \Xc^T \Xc ~ \preceq  ~\Xc^T \Yc~ &\preceq ~\lp \Xc^T \Xc, \label{eq:scal_cons} \\
    (\Yc - \lm \Xc)^T(\Yc - \lp \Xc) ~&\preceq ~0, \hspace{20mm} \label{eq:var_red} \\[-6.5mm]
  \end{align}
  where the notations $\succeq$ and $\preceq$ denote respectively positive and negative semi-definiteness.
  \item Constraints  \eqref{eq:eq_mean}, \eqref{eq:scal_cons}, \eqref{eq:var_red} are LMI Gram-representable.
\end{enumerate}
\end{theorem}
\begin{proof}
    First, we average elements from both sides of the assumption $Y^j = WX^j$ to obtain constraint \eqref{eq:eq_mean}:
    $$ \Yb_{j\cdot} = \frac{\mathbf{1}^T Y^j}{N} = \frac{\mathbf{1}^T WX^j}{N} = \frac{\mathbf{1}^T X^j}{N} = \Xb_{j\cdot} \text{ for $j=1,\dots,d$} $$
    where $\mathbf{1}^TW = \mathbf{1}^T$ follows from $W$ being generalized doubly stochastic, i.e., its rows and columns sum to one (see Definition \ref{def:gds}).

    In the sequel, we will consider all the components $j$ at once, and then we use the notation $Y=\Wt X$, where $\Wt = I_d \otimes W$ is in $\Rmat{Nd}{Nd}$. Since $\Wt$ is a block-diagonal matrix repeating $d$ times $W$, if $W\in \Wcl{\lm}{\lp}$, then $\Wt\in \Wcl{\lm}{\lp}$. \\
    The symmetry of the matrix $X^TY$ follows from the assumption $Y = \Wt X$, with $\Wt$ symmetric. The same argument shows the symmetry of $\Xc^T\Yc$, because $Y = \Wt X$ and $\Xb = \Yb$ imply $\Yc = \Wt \Xc$. \\
    Since the averaging matrix $\Wt$ is real and symmetric, we can take an orthonormal basis $\mathbf{v_1},\dots,\mathbf{v_{Nd}}$ of eigenvectors, corresponding to real eigenvalues $ \lam_{1}\ge\cdots \ge \lam_{Nd}$.
    Since $\Wt$ is composed of $d$ diagonal blocks of the generalized doubly stochastic matrix $W$, it has the same eigenvalues as $W$ but with a multiplicity $d$ times larger for each of them.
    Therefore, its largest eigenvalues are $\lam_j = 1$ ($j=1,\dots,d$) and corresponds to the eigenvectors $\mathbf{v_j} = [0\dots \mathbf{1}_N^T\dots 0]^T$ where the position for $\mathbf{1}_N$ corresponds to the position of the block $j$ in $\Wt$. These eigenvectors can be written in matrix form as $V_{1,\dots,d} = I_d \otimes \mathbf{1}_N $. By assumption, other eigenvalues are such that
    $$\lm \le \lam_i \le \lp \text{ for $i=d+1,\dots,Nd$, with $\lm, \lp \in (-1,1)$.} $$
    Let us now consider a combination $\Xc z$ of the columns of the matrix $\Xc$, for an arbitrary $z \in \Rvec{K}$. It can be decomposed in the eigenvector basis of $\Wt$, and used to express the combination $\Yc z$ as well: \\[-8mm]

    \small
    \begin{align} \label{eq:ev_basis}
      \hspace{-2mm} \Xc z = \sum_{i = d+1}^{Nd} \gamma_i \mathbf{v}_i, \text{ and }
      \Yc z = \Wt \Xc z = \sum_{i = d+1}^{Nd} \gamma_i \lambda_i \mathbf{v}_i, \hspace{4mm}
    \end{align}
    \normalsize
    where $\gamma_i$ are real coefficients. These coefficients for $\mathbf{v_1}, \mydots, \mathbf{v_d}$ are zero because $\Xc z$ is orthogonal to these eigenvectors associated with eigenvalue $ \lam_j = 1$. Indeed $\Xc$ is centered with respect to the agents, for any component $j$ or iteration k and then we have
  $$ (I_d \otimes \mathbf{1}_N)^T \Xc  = V_{1,\dots,d}^T \Xc = \mathbf{0}_{d\times K}.$$
    Using the decomposition \eqref{eq:ev_basis} to compute the scalar product $z^T \Xc \Yc z$ for any $z \in \Rvec{K}$ leads to the following scalar inequalities
    \begin{align}
      z^T \Xc^T \Yc z = \sum_{i = d+1}^{Nd} \gamma_i^2 \lambda_i &\ge  \lm z^T \Xc^T \Xc z, \\[-3mm]
                                                            &\le  \lp z^T  \Xc^T \Xc z.
    \end{align}
    Having these inequalities satisfied for all $z \in \Rvec{K}$, is equivalent to \eqref{eq:scal_cons}.
    In the same way, \eqref{eq:var_red} is obtained by verifying that the following inequality holds for all $z \in \Rvec{K}$:
    $$ (\Yc z - \lm \Xc z)^T(\Yc z - \lp \Xc z) \le 0. $$
    This can be done by substituting $\Xc z$ and $\Yc z$ using equation \eqref{eq:ev_basis}, and by using the bounds on $\lam_i$ ($i=d+1,\dots,Nd$). \\
    Finally, we prove part (iii) of the theorem. Constraint \eqref{eq:eq_mean} is linearly (and thus also LMI) Gram-representable because it can be expressed using only elements of $G_c$ \eqref{eq:Gc}, i.e. scalar product between $x_i^k $ and $y_j^k$,
    $$\footnotesize \Biggl(\frac{1}{N} \sum_i \qty(x_i^k-y_i^k) \Biggr)^T\Biggl(\frac{1}{N} \sum_j \qty(x_j^k-y_j^k) \Biggr) = 0 ~ \text{ for $k=1,...,K$.}$$
    Constraints \eqref{eq:scal_cons} and  \eqref{eq:var_red} are LMI Gram-representable because they are LMIs whose entries can be defined  using only the entries of $G_c$, which is a submatrix of the full Gram matrix $G$ of scalar products. For example, the entry $k,l$ of $\Xc^T\Xc$ can be expressed as the scalar product of columns $k$ and $l$ of $\Xc$
    $$ \qty(\Xc^T)_{k \cdot} \qty(\Xc)_{\cdot l} = \sum_{i=1}^N \Bigl(x_i^k - \frac{1}{N} \sum_j x_j^k \Bigr)^T \Bigl(x_i^l - \frac{1}{N} \sum_j x_j^l \Bigr). $$
\end{proof}
\smallskip
Using Theorem \ref{thm:conscons}, we can relax constraints \eqref{eq:cons_2} and \eqref{eq:cons_1} and replace them by \eqref{eq:eq_mean}, \eqref{eq:scal_cons} and  \eqref{eq:var_red}, which are LMI Gram-representable.
Then, Proposition \ref{prop:GramPEP} allows to write a relaxed SDP formulation of a PEP providing worst-case results valid for the entire spectral class of matrices $\Wcl{\lm}{\lp}$. We call this formulation the \textit{spectral PEP formulation} and its results the \textit{spectral worst-case}.
This SDP formulation has matrix $G\succeq0$ and vectors $f_{i,v}$ ($i=1,...,N$) as decision variables. The vector $f_{i,v}$ contains the function values of $f_i$ at the different iterates $x_i^k$ ($k=1,...,K$). The values in $G$ correspond to the scalar products of the iterates ($x_i^k, y_i^k$) and the gradients ($g_i^k$) of the different agents.

When different averaging matrices are used for different sets of consensus steps, the constraints from Theorem \ref{thm:conscons} can be applied independently to each set of consensus steps.

Constraint \eqref{eq:eq_mean} is related to the stochasticity of the averaging matrix and imposes that variable $x$ has the same agents average as $y$, for each consensus step and for each component.
Linear matrix inequality constraints \eqref{eq:scal_cons} and  \eqref{eq:var_red} imply in particular scalar constraints for the diagonal elements.
They correspond to independent constraints for each consensus step, i.e., for each column $\xc$ and $\yc$ of matrices $\Xc$, $\Yc$: \vspace{-0.5mm}
\begin{align}
  \lm \xc^T \xc \le \xc^T\yc &\le \lp \xc^T \xc, \label{eq:scal_cons_1}\\
  (\yc - \lm \xc)^T(\yc - \lp \xc) & \le 0. \label{eq:scal_cons_2} \\[-5.5mm]
\end{align}
These constraints imply in particular that \vspace{-0.5mm} $$\yc^T \yc \le \lmax^2~ \xc^T\xc,\quad \text{where $\lmax = \max(|\lm|,|\lp|)$,} \vspace{-0.5mm}$$
meaning that the disagreement between the agents, measured by $\yc^T \yc$ for $y$ and $\xc^T \xc$ for $x$, is reduced by a factor $\lmax^2 \in [0,1)$ after a consensus.
But constraints \eqref{eq:scal_cons} and  \eqref{eq:var_red} also allow linking different consensus steps to each other, \textit{via} the impact of off-diagonal terms, in order to exploit the fact that these steps use the same averaging matrix. \\
We can also interpret constraints \eqref{eq:scal_cons} and \eqref{eq:var_red} as a sum over all the dimensions $j=1,\dots,d$. Each term of this sum corresponds to the same constraint expression as \eqref{eq:scal_cons} and \eqref{eq:var_red}, but applied only to $\Xc^j$ and $\Yc^j \in \Rvec{N \times K}$.
Indeed, any product involving $\Xc$ or $\Yc \in \Rvec{Nd \times K}$ can be written as a sum over the dimensions $d$. For instance, for $\Xc^T \Xc$, we have \vspace{-1.5mm}
$$\Xc^T \Xc = \sum_{j=1}^d (\Xc^j)^T \Xc^j.$$

The following proposition show that constraint \eqref{eq:scal_cons} is redundant when we consider a symmetric range of eigenvalue, i.e., when $\lp = -\lm$.
\begin{proposition}[Symmetric range of eigenvalues] \label{prop:sym_ev}
  If $\lp = - \lm = \lam \in \qty[0,1]$, then LMI constraints \eqref{eq:scal_cons} and \eqref{eq:var_red} from Theorem \ref{thm:conscons} are equivalent to
  \begin{equation}
    \Yc^T\Yc \preceq \lambda^2 \Xc^T\Xc. \label{eq:var_red_sym}
  \end{equation}
\end{proposition}
\begin{proof} Constraint \eqref{eq:var_red} is equivalent to \eqref{eq:var_red_sym} when $\lp = - \lm = \lam$.
  Constraint \eqref{eq:scal_cons} with $\lp = - \lm = \lam$ becomes
  \begin{equation}
    -\lam \Xc^T \Xc ~ \preceq  ~\Xc^T \Yc~ \preceq ~\lam \Xc^T \Xc, \label{eq:scal_cons_sym}
  \end{equation}
  We now show that constraint \eqref{eq:scal_cons_sym} is implied by \eqref{eq:var_red_sym}, which achieves the proof.
  When $\lam \ge 0$, constraint \eqref{eq:var_red_sym} is equivalent to the following bound on $\|\Yc z\|$, for any $z\in \Rvec{K}$,
  \begin{equation} \label{eq:var_red_sym_eq}
    \|\Yc z\| \le \lam \|\Xc z\| \qquad \text{for any $z\in \Rvec{K}$.}
  \end{equation}
 Moreover, the scalar product between $\Xc z$ and $\Yc z$ can be bounded, for any $z\in \Rvec{K}$, using the Cauchy-Schwarz inequality
  \begin{equation} \label{eq:scal_cons_sym_eq}
     - \|\Xc z\| ~\|\Yc z\| \le (\Xc z)^T(\Yc z) \le \|\Xc z\| ~\|\Yc z\|.
  \end{equation}
  We can combine inequality \eqref{eq:scal_cons_sym_eq} with \eqref{eq:var_red_sym_eq} and obtain
  $$ - \lam z^T\Xc^T\Xc z \le z^T\Xc^T\Yc z \le \lam z^T \Xc^T\Xc z,~ \text{for any $z\in \Rvec{K}$,}$$
  which is equivalent to \eqref{eq:scal_cons_sym}.
\end{proof}
Proposition \ref{prop:sym_ev} means that when we consider a class of matrices with a symmetric range of eigenvalues, i.e., $\Wcl{-\lam}{\lam}$, we can remove constraint \eqref{eq:scal_cons} from the spectral PEP formulation, without modifying its result. Other constraints from Theorem \ref{thm:conscons}, including the symmetry of $X^TY$ and $\Xc^T\Yc$, should still be imposed.

\subsubsection*{Recovering the worst averaging matrix}
Theorem \ref{thm:conscons} provides necessary constraints for describing a set of consensus steps $Y^j=WX^j$ that use an unknown averaging matrix $W$ from a spectral class of symmetric generalized doubly-stochastic matrices $\Wcl{\lm}{\lp}$.
But these constraints are not known to be sufficient; so even if matrices $X$ and $Y$ satisfy constraints \eqref{eq:eq_mean}, \eqref{eq:scal_cons}, \eqref{eq:var_red}, we do not know if there is a matrix $W \in \Wcl{\lm}{\lp}$ such that $Y^j=WX^j$ for all $j=1,\dots,d$.
The question of the existence of such a matrix can be expressed as follows: is the optimal cost of the following problem equal to 0?
\begin{mini}{\Wh \in \Wcl{\lm}{\lp}}{\mynorm{\Yr-\Wh\Xr}_F\label{eq:SDP_Wh}}{}{}
\end{mini}
where matrices $\Xr, \Yr \in \Rmat{N}{Kd}$ stack the $X^j, Y^j \in \Rmat{N}{K}$ horizontally and thus reshape matrices $X$ and $Y$. This reshape is needed for recovering a matrix $\Wh$ which is identical for every dimension $j$ and that has appropriate size ($N \times N$).
Note that problem \eqref{eq:SDP_Wh} can easily be solved since it is an SDP, as the constraint $\Wh \in \Wcl{\lm}{\lp}$ can be formulated as $ {\lm I \preceq ( \Wh - \frac{\mathbf{11}^T}{N}) \preceq \lp I}$ together with $\Wh^T = \Wh$ and $\Wh \mathbf{1} = \mathbf{1}$.
If the optimal cost of \eqref{eq:SDP_Wh} is zero, the optimal value of $\Wh$ is a valid worst averaging matrix $W$.
Alternatively, problem \eqref{eq:SDP_Wh} without any constraints is a least-square problem whose solution is cheaper to compute and is given by ${\Whp = \Yr\Xr^{\dag}}$, where $\Xr^{\dag}$ is the pseudo-inverse of $\Xr$. If its remainder $\mynorm{\Yr-\Whp\Xr}_F$ is zero, then we check a posteriori that the constraint $\Whp \in \Wcl{\lm}{\lp}$ is satisfied. This was often the case in our experiments from Section \ref{sec:NumRes} and allows to recover rapidly the value of the worst averaging matrix. Note, though, that when $\Whp$ does not satisfy the required conditions, a valid $W$ might still exist, and so one must solve \eqref{eq:SDP_Wh} to find it.

We have now all the elements to write and solve PEPs for decentralized optimization methods, including ways of representing the consensus steps in a Gram-representable manner, allowing to formulate the PEP as an SDP. In the next section, we demonstrate the methodology by analyzing 3 decentralized methods.

\section{Demonstration of our methodology} \label{sec:NumRes}
Using results from previous sections, we can build two PEP formulations for analyzing the worst-case performance of a large class of decentralized optimization methods: the exact and the spectral formulations. These formulations allow to obtain rapidly and automatically accurate numerical performance bound.
We demonstrate the power of the methodology by focusing on the analysis of 3 selected algorithms.
For each algorithm, we use the same settings (performance criterion, initial condition, class of functions,...) as its theoretical bound, in order to obtain comparable results. These particular settings are not a limitation from our PEP formulations, which allows to represent a larger diversity of situations.
\subsection{Distributed (sub)gradient descent (DGD)} \label{sec:analysis_DGD}
We consider $K$ iterations of DGD described by \eqref{eq:DGD_cons} and \eqref{eq:DGD_comp}, with constant step-size $\alpha$, in order to solve problem \eqref{opt:dec_prob}, i.e., minimizing $f(x) = \frac{1}{N}\sum_{i=1}^N f_i(x)$, with $x^*$ as minimizer of $f$.
There are different studies on DGD; e.g., \cite{DGD1} shows that its iterates converge to a neighborhood of the optimal solution $x^*$ when the step-size is constant. In the sequel, we take as baseline the results of a recent survey \cite{DGD} providing a theoretical bound for the functional error at the average of all the iterates, valid when subgradients are bounded. \smallskip

\begin{theorem}[Performance of DGD {{\cite[Theorem 8]{DGD}}}] \label{thm:bound}
~\\ Let $f_i,\dots,f_N \in \mc{F}_R$, i.e. convex local functions with subgradients bounded by $R$. Let $x^0$ be an identical starting point for all agents such that $\|x^0 - x^*\|^2 \le D^2$. And let $W \in \Wcl{-\lam}{\lam}$ for some $\lam \in [0,1)$, i.e. $W$ is a symmetric and generalized doubly stochastic matrix with eigenvalues $\lam_2,\dots,\lam_N \in \qty[-\lam,\lam]$. \\
If we run DGD for $K$ steps with a constant step-size $\alpha = \frac{1}{\sqrt{K}}$, then there holds\footnote{Note the factor 2 in the second term of the bound \eqref{eq:th_bound} was missing in \cite{DGD} but its presence was confirmed by the authors of \cite{DGD}.}
\vspace{-3pt}
\begin{equation} \label{eq:th_bound}
  f(\xmoy ) - f(x^*) \le \frac{D^2 + R^2}{2 \sqrt{K}} + \frac{2 R^2}{\sqrt{K}(1-\lam)}, \vspace{-3pt}
\end{equation}
where $\xmoy = \frac{1}{N(K+1)} \sum_{i=1}^N \sum_{k=0}^K x_i^k $ is the average over all the iterations and all the agents.
\end{theorem}
Theorem \ref{thm:bound} was stated in \cite{DGD} for doubly stochastic matrix $W$ but the proof never uses the non-negativity assumption and therefore it also holds for generalized doubly stochastic matrices.
For comparison purposes, we will analyze the performance of DGD using our PEP formulations in exactly the same settings as Theorem \ref{thm:bound}.
Our first PEP formulation searches for solutions to the following maximization problem: \vspace{-2.2mm}
\begin{align}
  \hspace{-1cm}\max_{\substack{\footnotesize{~x^*, f_i, x^0_i,y_i^0\dots x^K_i,y_i^K} \\ \footnotesize{\text{for } i=1,\dots,N}}} ~ & \frac{1}{N}\sum_{i=1}^N \qty(f_i(\xmoy)- f_i(x^*)) &\tag{DGD$(W)$-PEP} \label{prob:DGD_PEP}\\
  \text{s.t.}  \hspace{12mm}    & \hspace*{-10mm}  f_i \in \mc{F}_R  &\small{\forall i} \\[-0.7mm]
       & \hspace*{-10mm} x^* = \underset{x}{\mathrm{argmin}}~ \frac{1}{N}\sum_{i=1}^N f_i(x),& \\[-0.8mm]
       & \hspace*{-10mm} \|x^0_i - x^*\|^2 \le D^2 \quad \text{and} \quad x_i^0 = x_j^0 & \small{\forall i, j }\\[-0.8mm]
       & \hspace*{-10mm} y_i^k = \sum_{j=1}^N w_{ij} x_j^k, & \small{\forall i, k} \label{eq:DGD-PEP-cons}\\[-0.7mm]
       & \hspace*{-10mm} x_i^{k+1} = y_i^k - \alpha^k \nabla f_i(x_i^k), & \small{\forall i, k} \\[-5mm]
\end{align}
The objective function and the constraints of \eqref{prob:DGD_PEP} are all linearly Gram-representable and the problem can then be formulated as an SDP, according to Proposition \ref{prop:GramPEP}.
This is referred to as the \emph{exact formulation} because it finds the exact worst-case performance of the algorithm for a specific given matrix $W$.
The second formulation relaxes the consensus constraints imposed by equation \eqref{eq:DGD-PEP-cons} and replaces them with the constraints from Theorem \ref{thm:conscons}, with $-\lm=\lp = \lam$. Those are LMI Gram-representable (see Theorem \ref{thm:conscons}) and can then be used in the SDP formulation of PEP, according to Proposition \ref{prop:GramPEP}. This formulation is referred to as the \emph{spectral formulation} and provides \emph{spectral worst-cases},
i.e., upper bounds on the worst-case performances of the algorithm, valid for any matrix $W \in \Wcl{-\lam}{\lam}$. These spectral bounds can thus be compared with the bound from Theorem \ref{thm:bound}.

In our experiments, we focus on the situation where $D = 1$ and $R = 1$, but the results obtained can be scaled up to general values using changes of variables, see Appendix \ref{annexe:scaling}. \smallskip

\paragraph{Impact of the number of agents $N$}
In Fig. \ref{fig:wc_Nevol}, we observe that the results of the spectral formulation are \emph{independent} of the number of agents $N \ge 2$ in the problem. This is shown for $K$ = 5 iterations and different spectral ranges.
This observation has been confirmed for other values of $K$ (10, 15, and 20).
The theoretical performance bound from Theorem \ref{thm:bound} is also independent of $N$. Therefore, in the sequel, we analyze the spectral formulation for $N=3$, which keeps the computational complexity low while keeping some non-trivial network matrices.

\begin{figure}[h!]
  \vspace{-1mm}
  \centering
  \includegraphics[width=0.5\textwidth]{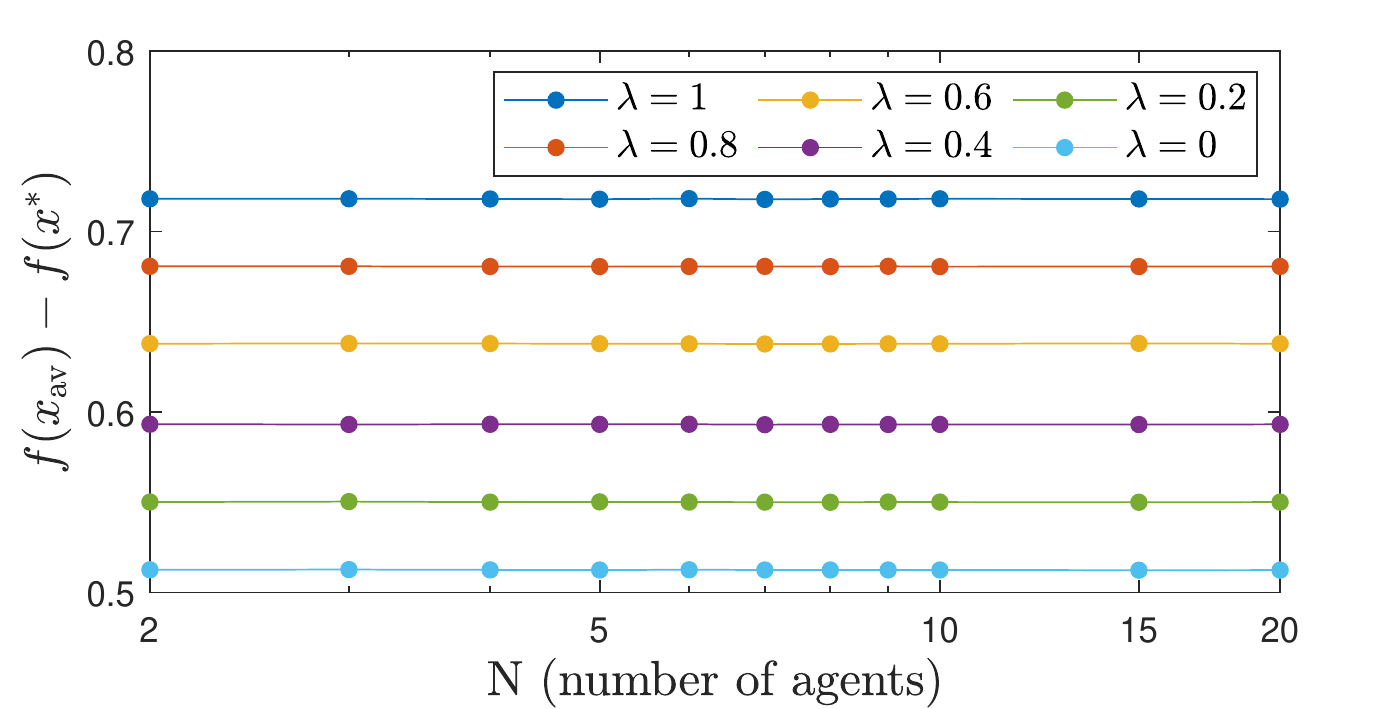}
    \caption{Independence of $N$ for the spectral worst-case performance of $5$ iterations of DGD in the setting of Theorem \ref{thm:bound}. \vspace{-4mm}}
  \label{fig:wc_Nevol}
\end{figure}

\paragraph{Comparison with Theorem \ref{thm:bound}}
We compare the spectral bound with the theoretical bound from Theorem \ref{thm:bound} for different ranges of eigenvalues $\qty[-\lam,\lam]$ for the network matrix. The value of $1-\lam$ corresponds to the algebraic connectivity of the network.
Fig. \ref{fig:wc_lamevol_N3} shows the evolution of both bounds with $\lam$ for $K = 10$ iterations of DGD with $N = 3$ agents. We observe that the spectral worst-case performance bound (in blue) largely improves on the theoretical one (in red), especially when $\lam$ approaches 1, in which case the theoretical bound grows unbounded.

\begin{figure}[h!]
  \centering
  \includegraphics[width=0.5\textwidth]{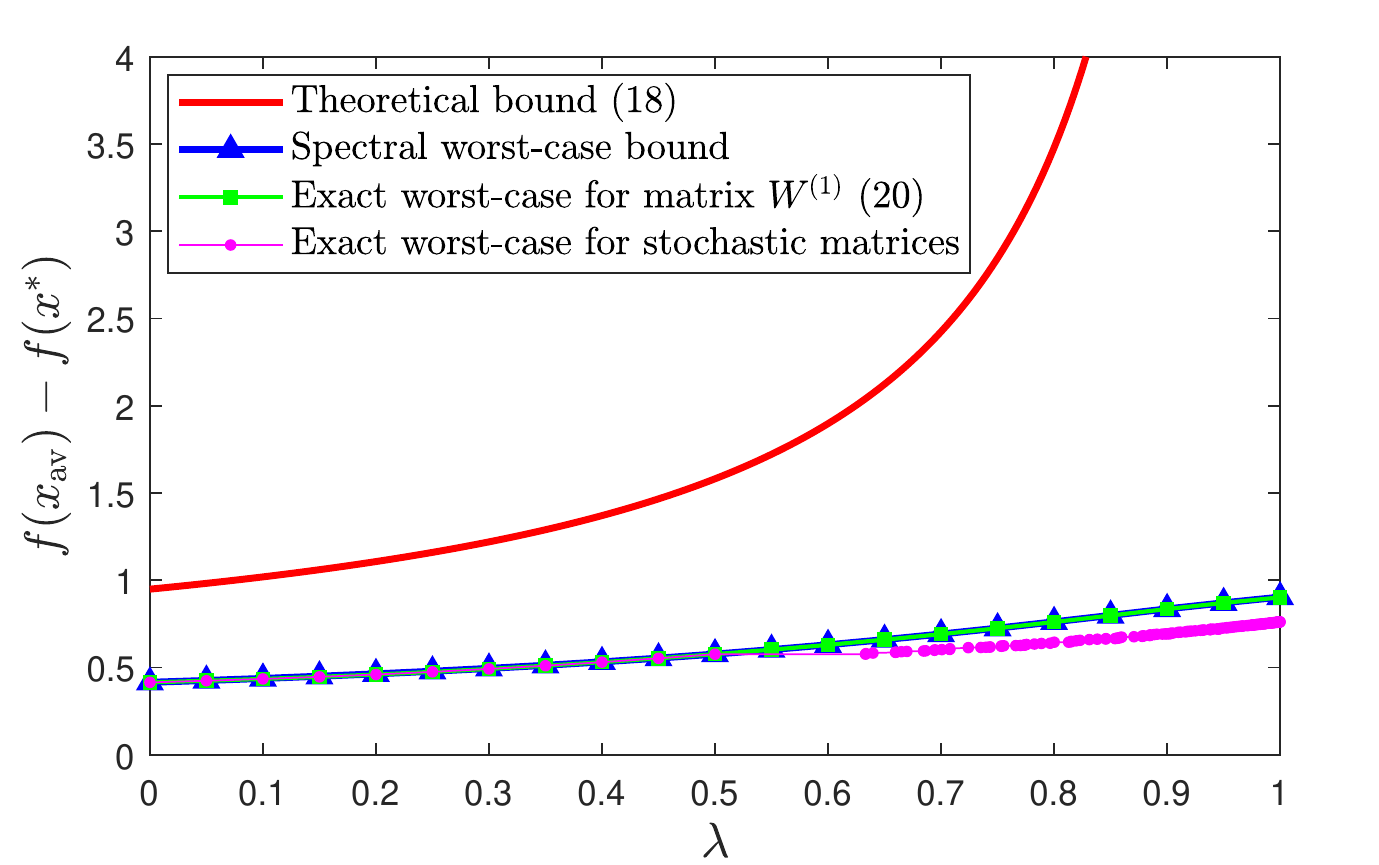}
  \vspace{-3mm}
  \caption{Evolution with $\lam$ of the worst-case performance of $K = 10$ iterations of DGD in the setting of Theorem \ref{thm:bound} with $N = 3$ agents. The plot shows (i) the theoretical bound from equation \eqref{eq:th_bound} (in red), largely above (ii) the spectral worst-case performance (in blue), (iii) the exact worst-case performance for the symmetric generalized doubly stochastic matrix $W^{(1)}$ from equation \eqref{eq:mat} (in green) and (iv) the exact worst-case performance for symmetric doubly stochastic matrices found based on an exhaustive exploration of such matrices used in the exact PEP formulation (in pink).
  This indicates the tightness of the spectral formulation of PEP for DGD with symmetric generalized doubly stochastic matrices, within numerical errors.
  \vspace{-2mm}}
  \label{fig:wc_lamevol_N3}
\end{figure}

The improvements of the bounds when $\lam$ is close to 1 is particularly relevant since large values for $\lam$ are frequent for averaging matrices of large networks of agents \cite{eigenBound}.
For example, for a 5 by 5 grid of agents with Metropolis weights \cite{DGD}, the range of eigenvalues of the resulting averaging matrix is $\qty[-0.92,0.92]$.
In that case, after $K=10$ iterations, our spectral bound guarantees that the performance measure is below $0.85$, compared to 8.2 for the theoretical bound from Theorem \ref{thm:bound}.
This accuracy of $0.85$ would only be guaranteed using Theorem \ref{thm:bound} with $K = 936$. \smallskip

\paragraph{Worst averaging matrix and tightness analysis}
When considering the spectral formulation with a symmetric spectral range $-\lm=\lp=\lam$, we observe that the worst averaging matrices are matrices of the following form \vspace{-0.3mm}
\begin{equation} \label{eq:mat}
  W^{(1)} = J - \lam (I - J), \vspace{-0.3mm}
\end{equation}
where $J = \frac{\mathbf{11^T}}{N}$, i.e. it has all entries equal to $\frac{1}{N}$, and $I$ is the identity matrix. Matrix $W^{(1)}$ is symmetric and generalized doubly stochastic, leading $1$ to be one of its eigenvalues. All its other eigenvalues are equal to $-\lam$.
Matrix $W^{(1)}$ always produces a remainder $\|\Yr - W^{(1)}\Xr\|_F$ close to zero for DGD, but it may not be the only one.  The bound obtained using the exact PEP formulation with this specific matrix $W^{(1)}$ for $K=10$ is plotted in green in Fig. \ref{fig:wc_lamevol_N3} and \emph{exactly} matches the spectral bound in blue, within numerical errors. This means that the spectral formulation provides a \emph{tight performance bound} for DGD with symmetric generalized doubly stochastic matrices, even though it is a relaxation,.
This observation has been confirmed for other values of $K$ (5, 15, and 20). \smallskip

\paragraph{Doubly stochastic versus generalized doubly stochastic}
Since every doubly stochastic matrix is also generalized doubly stochastic, the spectral bound also provides an upper bound on the performance of DGD with symmetric doubly stochastic matrices. This bound remains tight for $\lam \le \frac{1}{N-1}$ because the worst-case matrix $W^{(1)}$ \eqref{eq:mat} we have obtained is non-negative and is therefore doubly stochastic. For $\lam > \frac{1}{N-1}$, this is no longer the case and the analysis is performed by empirically looking for symmetric stochastic averaging matrices leading to the worst performance.
In Fig. \ref{fig:wc_lamevol_N3}, for $N=3$ and $\lam>0.5$, we have generated more than 6000 random symmetric doubly stochastic 3 by 3 matrices. We have analyzed their associated DGD performance using the exact PEP formulation and have only kept those leading to the worst performances.
The resulting pink curve deviates no more than 20\% below the spectral bound. In that case, the spectral bound is thus no longer tight for DGD with doubly stochastic matrices but remains very relevant.
This observation has been confirmed for other values of $K$ and $N$ ($N = 3,5,7$, and $K=10,15$). \smallskip

\paragraph{Evolution with the total number of iterations $K$}
Fig. \ref{fig:wc_Kevol} shows the evolution of the spectral worst-case performance for DGD multiplied by $\sqrt{K}$, for different values of $\lam$. Except when $\lam = 1$, all lines tend to a constant value, meaning that the spectral bound behaves in $\bigO\qty(\frac{1}{\sqrt{K}})$, as the theoretical bound \eqref{eq:th_bound}, but with a much smaller multiplicative constant.
When $\lam = 1$, the line grows linearly and never reaches a constant value. In that case, the worst averaging matrices lead to counterproductive interactions, preventing DGD from working in the worst case. \smallskip

\paragraph{Tuning the step-size $\alpha$}
The PEP methodology allows us to easily tune the parameters of a method.
For example, Fig. \ref{fig:wc_alphevol} shows the evolution of the spectral worst-case performance of DGD with the constant step-size it uses,
\begin{figure}[h!]
  \vspace{-1mm}
  \centering
  \includegraphics[width=0.48\textwidth]{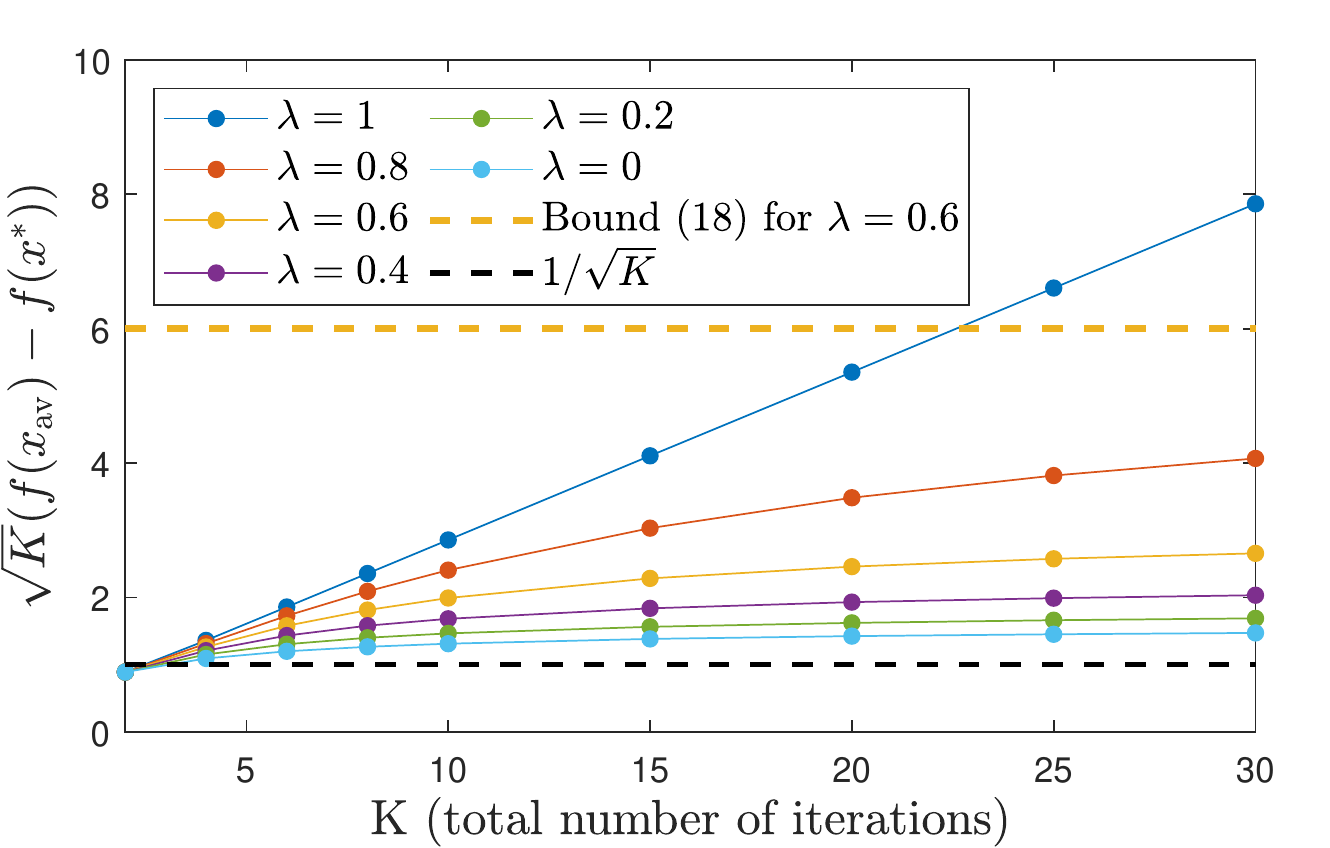}
  \caption{Evolution with $K$ of the \emph{normalized} spectral worst-case performance of $K$ iterations of DGD in the setting of Theorem \ref{thm:bound} with $N = 3$. The shown spectral worst-cases are normalized by $\frac{1}{\sqrt{K}}$ to show that they evolve at this rate.
  \vspace{-4mm}}
  \label{fig:wc_Kevol}
\end{figure}
\begin{figure}[b]
  \vspace{-3mm}
  \centering
  \includegraphics[width=0.48\textwidth]{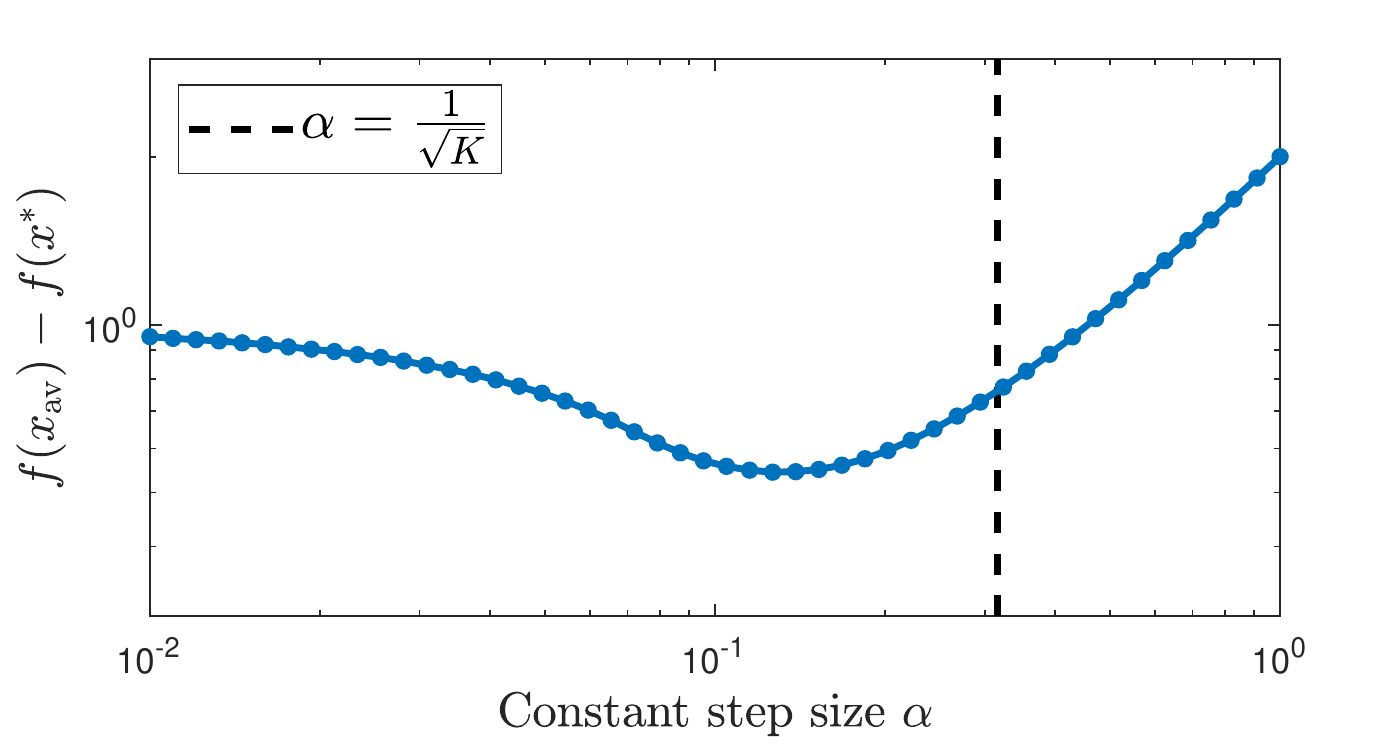}
  \vspace{-1mm}
  \caption{Evolution with $\alpha$ of the spectral worst-case performance of $K = 10$ iterations of DGD in the setting of Theorem \ref{thm:bound} with $N$ = 3 agents and $\lam = 0.8$ (except for $\alpha$).}
  \label{fig:wc_alphevol}
\end{figure}
in the setting of Theorem \ref{thm:bound} with $N=3$, $K=10$ and $\lam = 0.8$.
In that case, we observe that the value $\alpha = \frac{1}{\sqrt{K}}$ used in Theorem \ref{thm:bound} for deriving the theoretical performance bound is not the best possible choice for $\alpha$ and should be divided by two to improve the performance guarantees by 30\%.
The optimal value for $\alpha$, regarding our spectral bound, is the one that provides the best worst-case guarantee,
whatever the averaging matrix from $\Wcl{-\lam}{\lam}$ is used.

The impact of the step-size on the other experiments and observations can be studied by setting $\alpha = \frac{h}{\sqrt{K}}$, for some $h > 0 $. We focused on $h=1$ for comparison with the theoretical bound from Theorem \ref{thm:bound}. Nevertheless, all our other observations have been confirmed\footnote{For too small step-size such as $h \le 0.1$, the worst averaging matrix observed is no more $W^{(1)}$.} for $h = 0.1, 0.5, 2, 10$. \vspace{-5mm}

\subsection{DIGing}
The DIGing algorithm \cite{DIGing}, described in Algorithm \ref{algo:DIGing}, combines DGD with a \emph{gradient tracking} technique. Each agent $i$ holds an estimation $\yDIG_i$ of the average of all the local gradients and uses it in its update instead of its own local gradient.
The DIGing algorithm allows using a different network matrix $W^k$ at each iteration $k$.
\begin{algorithm}
  \caption{DIGing}
  \begin{algorithmic}
  \STATE Choose step-size $\alpha > 0$ and pick any $x_i^0 \in \Rvec{d}$;
  \STATE Initialize $\yDIG^0_i = \nabla f_i(x_i^0)$ for all $i=1,\dots,N$;
  \FOR{$k = 0, 1,\dots$}
  \FOR{$i=1,\dots,N$} \vspace{1mm}
    \STATE  $x_i^{k+1} = \sum_{j=1}^N w_{ij}^k x_j^k - \alpha \yDIG_i^k$;\vspace{1mm}
    \STATE  $\yDIG_i^{k+1} = \sum_{j=1}^N w_{ij}^k  \yDIG_j^k + \nabla f_i(x_i^{k+1}) - \nabla f_i(x_i^k)$; \vspace{1mm}
  \ENDFOR
  \ENDFOR
  \end{algorithmic}
  \label{algo:DIGing}
\end{algorithm}

The linear convergence of DIGing has been established in \cite[Theorem 3.14]{DIGing}
provided that the local functions $f_i$ are $L$-smooth and $\mu$-strongly convex (with $L \ge \mu > 0$); the network matrices are symmetric, doubly stochastic, and have their second largest eigenvalue $\lam$ below 1 (in absolute value); and the step-size is within the interval \vspace{-5mm}
$$\alpha \in \Bigg(0, \frac{(1-\lam)^2}{2L\qty(1+4\sqrt{N}\sqrt{L/\mu})} \Bigg]. \vspace{-5mm} $$
The largest accepted step-size decreases thus as $\bigO(\frac{1}{\sqrt{N}})$ and has also a dependence in $\bigO(\frac{1}{L\sqrt{L/\mu}})$ which is less favorable than the usual $\bigO(\frac{1}{L})$ in optimization, leading thus often to very small values for accepted $\alpha$.
The spectral condition on the network matrices ($|\lam| < 1$) guarantees the network connectivity at each iteration. Actually, \cite{DIGing} imposes a weaker spectral condition which only requires that the union of all the networks over $B$ steps is connected. In this section, we consider the case $B=1$. Under all these conditions, \cite[Theorem 3.14]{DIGing} guarantees the following \mbox{R-linear}\footnote{\label{note:defConv} We recall the definition of the R-linear convergence and its differences with the Q-linear convergence, based on the definitions provided in \cite{DIGing}. \\ Suppose that a sequence $\{x^k\}$ converges to $x^*$ in some norm $\|\cdot\|$.
We say that the convergence is: (i) R-linear if there exists $ \rho \in (0, 1)$ and some positive constant $C$ such that $\|x^k - x^*\| \le C \rho^k$ for all $k$; (ii) Q-linear if there exists $\rho \in (0, 1)$ such that $ \frac{\|x^{k+1} - x^*\|}{\|x^{k} - x^*\|} \le \rho$ for all $k$.
Both convergences are geometric but the Q-linear convergence is stronger since it implies monotonic decrease of $\|x^k - x^*\|$, while R-linear convergence does not. By definition, Q-linear convergence implies R-linear convergence with the same rate but the inverse implication does not hold in general.} convergence:
\begin{equation} \label{eq:DIGing_conv}
  \sqrt{\sum_{i=1}^N \|x_i^K-x^*\|^2} \le C \rho_{\mathrm{theo}}^K, \quad \text{ for any $K \in \mathbb{N}$,}
\end{equation}
where $C$ is a positive constant and $\rho_{\mathrm{theo}} \in \qty(0,1)$ is the convergence rate depending on $N$, $\lam$, $L$ and $\mu$ (see \cite[Theorem 3.14]{DIGing} for details about its expression).

This section analyzes the worst-case performance of DIGing via the exact and spectral formulations, in the same settings as \cite[Theorem 3.14]{DIGing} to get a fair comparison. Therefore, we consider the set of $L$-smooth and $\mu$-strongly convex functions for local functions.
As performance criterion, we consider the same as in \eqref{eq:DIGing_conv} but squared and scaled by $N$: \vspace{-2mm}
\begin{equation} \label{eq:DIGing_crit}
  \mc{P}^K = \frac{1}{N}\sum_{i=1}^N \|x_i^K-x^*\|^2. \vspace{-1mm}
\end{equation}
The corresponding theoretical convergence rate for this criterion is given by $ \rho_{\mathrm{theo}}^2$.
We also consider initial conditions similar to those used implicitly in \cite[Theorem 3.14]{DIGing}:
\begin{align}
  \frac{1}{N}\sum_{i=1}^N\|x_i^0 - x^*\|^2 \le D^2, \label{eq:init_DIGing1} \\
  \frac{1}{N}\sum_{i=1}^N \|\yDIG_i^0 - \frac{1}{N}\sum_{j=1}^N\nabla f_j(x_j^0)\|^2 \le E^2.  \label{eq:init_DIGing2}
\end{align}
Condition \eqref{eq:init_DIGing1} bounds the initial performance criterion $\mc{P}(x^0)$ which measures the average error of the initial iterates $ x_i^0$. Condition \eqref{eq:init_DIGing2} bounds the average error made by the agents on the initial average gradient estimates $\yDIG_i^0$. \\
For the spectral formulation, to have the same setting as \cite[Theorem 3.14]{DIGing}, we consider time-varying averaging matrices that are symmetric, generalized doubly stochastic, and with a symmetric range of eigenvalues $[-\lam, \lam]$, i.e. in $\Wcl{-\lam}{\lam}$. In a second time, we will also consider constant network matrices.

The problem depends on 6 parameters: $L$, $\mu$, $D$, $E$, $\lam$, $\alpha$. We fix $L=1$ and $D=1$, as the results can then be scaled up to general values using appropriate changes of variables. The value of $E$ is arbitrarily fixed to $E = 1$, but all the observations have been confirmed for other values of $E$ ($E = 0.1, 10$). Different values of the step-size $\alpha$ will be analyzed to understand its impact on the worst-case performance of the DIGing algorithm. We show the results for representative values of $\mu$ and $\lam$ ($\mu = 0.1$ and $\lam = 0.9$).  \smallskip

\paragraph{Impact of the number of agents $N$}
As it was the case for DGD (see Section \ref{sec:analysis_DGD}), we have observed that the spectral worst-case of DIGing is independent of the number of agents $N$ (for $N \ge 2$). This differs from the theoretical analysis of DIGing from \cite[Theorem 3.14]{DIGing} for which the range of accepted step-sizes, as well as the convergence rate, depend on $N$. It appears that the worst-case performance of DIGing does not get worst when $N$ increases, or can at least be bounded uniformly for all values of $N$, for example with our spectral bound. Such uniform bound will allow us to better choose the step-size $\alpha$ of DIGing, identically for all $N$.
For the subsequent analysis of the results of our spectral PEP formulation for DIGing, we fix $N=2$ for keeping a low computational complexity. This also corresponds to the most favorable situation for the theoretical results \cite[Theorem 3.14]{DIGing}, which get worse as $N$ increases. \smallskip

\begin{figure}[b]
  \vspace{-3mm}
  \includegraphics[width=0.5\textwidth]{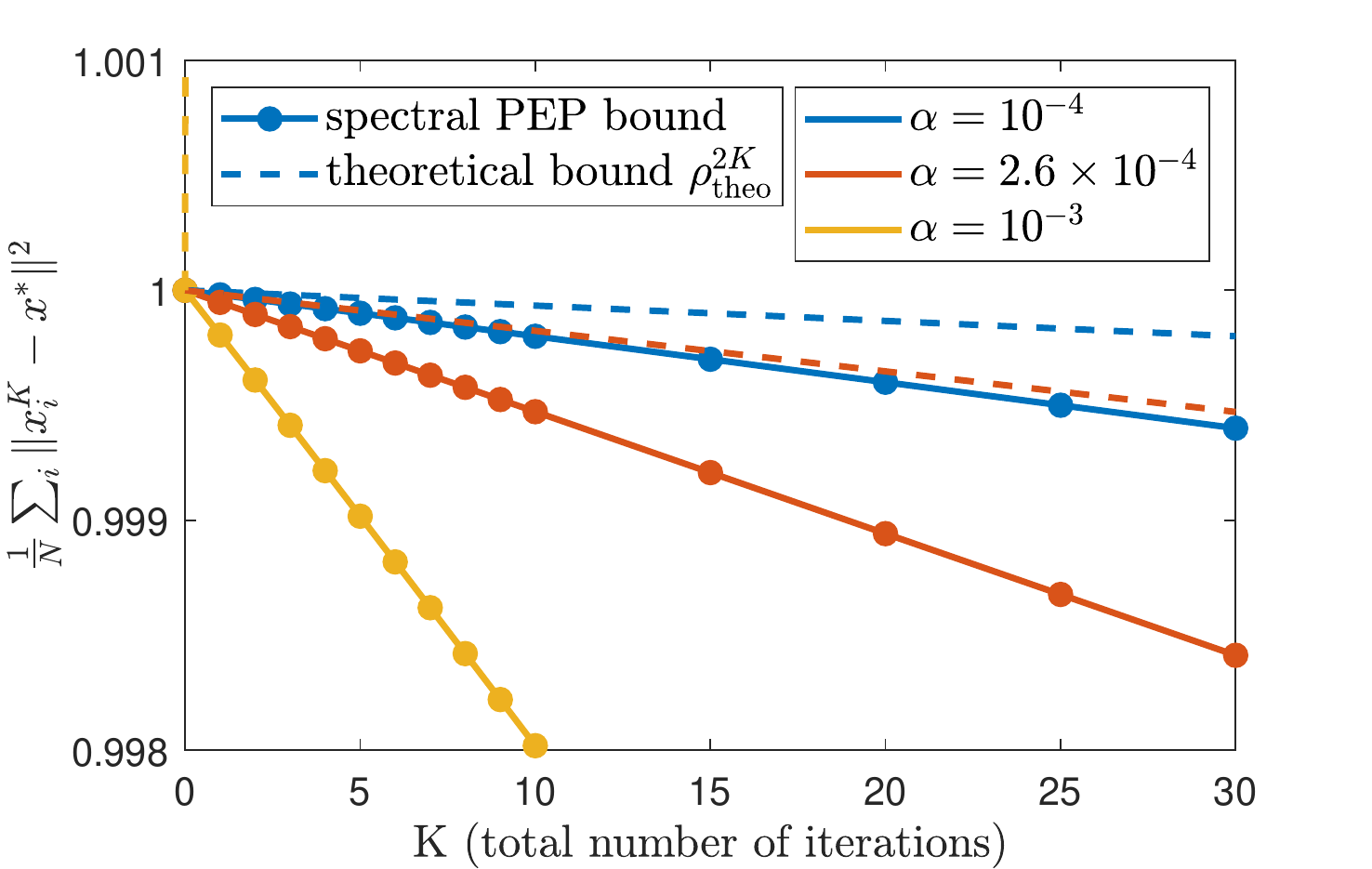}
  \caption{Evolution with K of the spectral worst-case of K iterations of DIGing with $N=2$, $\lam = 0.9$, $\mu = 0.1$, and different values of step-size $\alpha$. The corresponding theoretical rates from \cite[Theorem 3.14]{DIGing} are also shown in comparison. Logarithmic y-axis.}
  \label{fig:DIGing_Kevol}
\end{figure}

\paragraph{Comparison between the spectral and theoretical bounds}
We compare the spectral bounds obtained with PEP with the corresponding guarantees obtained using $\rho_{\mathrm{theo}}^2$, i.e. the square of the theoretical convergence rate bound \cite[Theorem 3.14]{DIGing}. Fig. \ref{fig:DIGing_Kevol} shows the evolution of both guarantees (spectral and theoretical) with the total number of iterations $K$ of the algorithm, for different values of the step-size $\alpha$ and for $N=2$, $\lam = 0.9$ and $\mu = 0.1$. The spectral bounds are always smaller than the corresponding theoretical ones.

For the three values of $\alpha$, we observe a linear decrease of the spectral worst-cases, which strongly suggests a linear convergence rate. The observed rates are listed in TABLE \ref{tab:rates_DIG} below and can be compared with the theoretical convergence rate $\rho_{\mathrm{theo}}^2$, which are all larger. The step-size $\alpha = 2.6 \times 10^{-4}$ is the one that optimizes the theoretical convergence guarantee $\rho_{\mathrm{theo}}$ from \cite[Theorem 3.14]{DIGing} for $N=2$.
For some step-sizes, such as $\alpha = 10^{-3}$, convergence is not guaranteed by \cite[Theorem 3.14]{DIGing}, even when $N=2$, while our observations suggest that it does occur. And this get worse when $N$ becomes larger since it would require the theoretical step-sizes to be smaller, i.e. $\alpha_{\max} \approx  \frac{4 \times 10^{-4}}{\sqrt{N}}$.
Therefore, the spectral bound for DIGing can help to greatly improve the choice of the step-size.

The same observations can be done for other values of $\mu$ and $\lam$. In particular, other values of $\mu$ leads to the same graphs, with a different scale for the vertical axis. We test it for $\mu = 0.01$ and $\mu = 0.001$.

\begin{table}[h]
  \vspace{1mm}
\centering
\captionsetup{width=.9\linewidth}
\begin{tabular}{|r|c|c|}
  \hline && \\[-2mm]
step-size $\alpha$  & $1- \text{observed rate}$ & $1-\text{theoretical rate}$ \\[1mm] \hline && \\[-2.5mm]
$\hspace{2.2mm} 10^{-4}$   & $2 \times 10^{-5}$ & $7\times 10^{-6}$ \\
$2.6 \times 10^{-4}$ & $5 \times 10^{-5}$ & $2\times 10^{-5}$ \\
$\hspace{2.2mm} 10^{-3}$   & $2 \times 10^{-4}$ &  / \\ \hline
\end{tabular}
\caption{Theoretical \cite{DIGing} and observed spectral rates for $N=2$, $\lam = 0.9$, $\mu = 0.1$ and for different step-sizes $\alpha$.\vspace{-4mm}}
\label{tab:rates_DIG}
\end{table}

\paragraph{Convergence rate analysis with PEP}
Fig. \ref{fig:DIGing_Kevol} strongly suggests a linear convergence rate but the performance guarantees only hold for the values of $K$ tested and we cannot extrapolate them with certainty, e.g. the performance could explode after a larger number of iterations.
We now show how to use our PEP formulations to obtain a \emph{guaranteed convergence rate} valid for any number of iterations. The idea is to use the same metric as an initial condition and performance measure to be able to consider the problem over only one general DIGing iteration. We also need to ensure that the DIGing update preserved the assumptions on initial conditions.
The metric we use is the weighted combination of the two error measures \eqref{eq:init_DIGing1} and \eqref{eq:init_DIGing2} previously used separately in the initialization: \vspace{-4mm}

{\small
\begin{align}
   \mc{P}_\betac^K = \frac{1}{N}\sum_{i=1}^N \|x_i^K-x^*\|^2 + \frac{\betac}{N}\sum_{i=1}^N \|\yDIG_i^K- \frac{1}{N} \sum_{j=1}^N \nabla f_j(x_j^K)\|^2, \label{eq:rate_metric}\\[-10mm]
\end{align}
\normalsize}
where $\betac$ is a positive weighting coefficient. \smallskip

\begin{proposition}[Convergence rate of DIGing with PEP] \label{thm:dig_conv}
  Consider the one iteration spectral PEP formulation of DIGing with $\mc{P}_\betac^1$ as performance criterion and with the following initialization: pick any $x_i^0, s_i^0 \in \Rvec{d}$ such that \vspace{-1mm}
  \begin{align}
    \sum_{i=1}^N \yDIG_i^0 &= \sum_{i=1}^N \nabla f_i(x_i^0) \qquad \text{ and } & \mc{P}_\betac^0 &= 1. \label{eq:init_rate} \vspace{-1mm}
  \end{align}
  Let $\theta_\betac$ be the optimal value of this PEP, then
  \begin{equation}
    \mc{P}^k \le \mc{P}_\betac^k \le \theta_\betac^k \mc{P}_\betac^0 \qquad \text{for any $k$, $\betac \ge 0$.} \vspace{-1mm}
  \end{equation}
  Convergence is Q-linear\footnote{See definition of Q-linear and R-linear convergence in footnote \ref{note:defConv}} for $\mc{P}_\betac^k$ \eqref{eq:rate_metric} and R-linear for $\mc{P}^k$\,\eqref{eq:DIGing_crit}, both with convergence rate $\theta_\betac$ depending on coefficient $\betac$.
\end{proposition}
\begin{proof}
  One can verify that the following changes of variables, using a coefficient $M \ge 0$,
  $$\tilde{x}_i = \sqrt{M} x_i, \quad \tilde{s}_i = \sqrt{M} s_i \quad \text{ and } \quad \tilde{f}_i(\tilde{x}_i) = M f_i(x_i),$$
  do not affect the behavior of DIGing and scale both $\mc{P}_\betac^0$ and $\mc{P}_\betac^1$ by a factor $M$:
  $$\tilde{\mc{P}}_\betac^0 = M \mc{P}_\betac^0, \qquad \tilde{\mc{P}}_\betac^1 = M \mc{P}_\betac^1.$$
  Since $\theta_\betac$ is the optimal value of $\mc{P}_\betac^1$ and $M = \tilde{\mc{P}}_\betac^0$ (for $\mc{P}_\betac^0 = 1$), we have that \vspace{-1mm}
  \begin{equation} \label{eq:scaling_ineq}
    \tilde{\mc{P}}_\betac^1 \le \theta_\betac \tilde{\mc{P}}_\betac^0.
  \end{equation}
  Equation \eqref{eq:scaling_ineq} holds for any value of $\tilde{\mc{P}}_\betac^0 \ge 0$ (e.g. for $\mc{P}_\betac^k$). Moreover, the iterations of DIGing are independent of $k$, and thus, inequality \eqref{eq:scaling_ineq} is valid for any iteration $k$
  \begin{equation} \label{eq:scaling_ineqk}
    \mc{P}_\betac^{k+1} \le \theta_\betac \mc{P}_\betac^k,
  \end{equation}
  provided that iterates $x^k_i, s_i^k$ also satisfy the initial condition
  \begin{align}
    \sum_{i=1}^N \yDIG_i^k &= \sum_{i=1}^N \nabla f_i(x_i^k), & \text{for any $k$}. \label{eq:dig_proof_rec2}
  \end{align}
  This condition \eqref{eq:dig_proof_rec2} holds by assumption for $k=0$ and is preserved by a DIGing update with a stochastic matrix $W$ (see Algorithm \ref{algo:DIGing}), as \vspace{-1mm}
  \small
  \begin{align}
    \sum_{i=1}^N \yDIG_i^{k+1} &= \sum_{i=1}^N \sum_{j=1}^N w_{ij}^k  \yDIG_j^k + \sum_{i=1}^N \nabla f_i(x_i^{k+1}) -  \sum_{i=1}^N \nabla f_i(x_i^k) \\
    & =
    \sum_{j=1}^N  \yDIG_j^k + \sum_{i=1}^N \nabla f_i(x_i^{k+1}) -  \sum_{i=1}^N \nabla f_i(x_i^k) \\
    &= \sum_{i=1}^N \nabla f_i(x_i^{k+1}), \vspace{-1mm}
  \end{align}
  \normalsize
  Finally, by definition of $\mc{P}^k$ (see \eqref{eq:DIGing_crit}) and $\mc{P}_\betac^k$, we have well that $\mc{P}^k \le \mc{P}_\betac^k$ for any $k$, $\betac \ge 0$.
\end{proof}

Using Proposition \ref{thm:dig_conv}, we can obtain guaranteed convergence rates for DIGing, which depend on the weighting coefficient $\betac$, for the metric $\mc{P}_\betac^K$.
These convergence rates are also valid for $\mc{P}^K$, for all $\betac \ge 0$, and can thus be compared with the observed rates from Fig. \ref{fig:DIGing_Kevol} and the theoretical convergence rates \eqref{eq:DIGing_conv} \mbox{(\cite[Theorem 3.14]{DIGing})}.
An exploration of the different values for the weighting coefficient $\betac \ge 0$ suggests that the best rates are obtained for $\betac = \frac{\alpha}{L}$ but other rates are also valid.
With $\betac=\frac{\alpha}{L}$, we recover exactly the same rates as those observed in Fig. \ref{fig:DIGing_Kevol}, but they are guaranteed with certainty for any number of iterations $K$. The PEP problem from Proposition \ref{thm:dig_conv} has a small size since it only considers one iteration, and is thus rapidly solved. The size of the problem still increases with the number of agents $N$ but once again, we observe that the results are independent of $N$.

The approach above can be applied to other algorithms provided that their updates are independent of each other.
It presents some parallels with the approach used in the automatic analysis with IQC \cite{IQC_dec}. Both approaches analyze the decrease of a particular function over only one iteration. We design the decreasing criterion $\mc{P}_\betac$ by hand and have optimized the value of $\betac$ to find the smallest rate. The IQC approach allows to easily optimize the rate over a wider class of Lyapunov functions and it may therefore give smaller rates. On the other hand, our PEP approach provides the worst functions and communication networks resulting from the worst-case solutions. It also allows comparing what happens over one iteration and over several, and with network matrices that are variable or constant in time. \smallskip

\begin{figure}[t]
  \vspace{-3mm}
  \centering
  \includegraphics[width=0.5\textwidth]{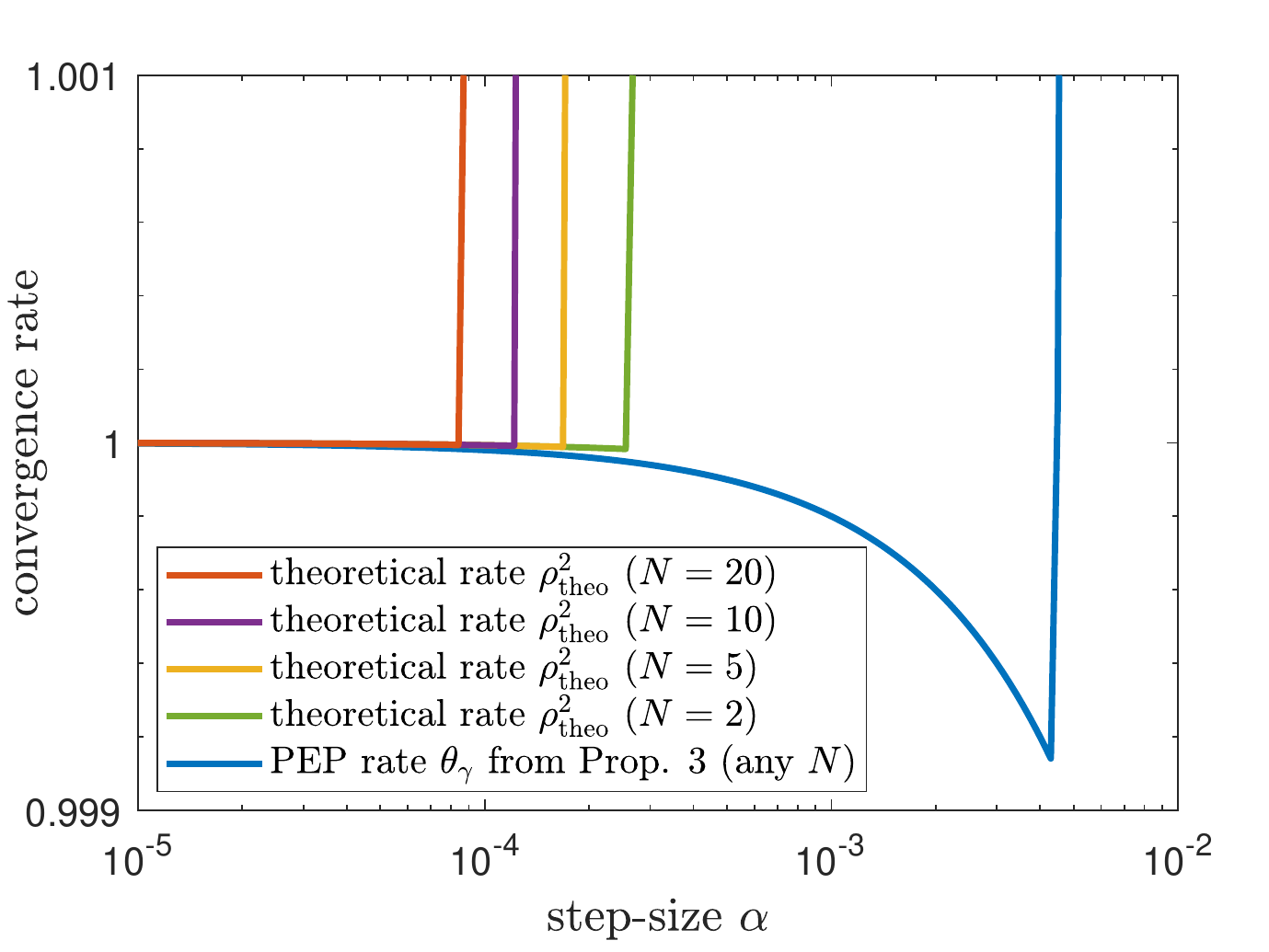}
  \caption{Convergence rate evolution with the step-size $\alpha$ for DIGing with $\lam=0.9$ and $\mu=0.1$. Theoretical rates from \cite{DIGing} are shown for different values of $N$. PEP rates $\theta_\betac$, obtained with Proposition \ref{thm:dig_conv} for $\betac = \frac{\alpha}{L}$, are lower and allow for larger step-sizes. These rates are computed with $N=2$ but the results are identical for any value of $N$. \vspace{-4mm}}
  \label{fig:DIGing_alpha_evol}
\end{figure}

\paragraph{Impact of the step-size $\alpha$}
Fig. \ref{fig:DIGing_alpha_evol} compares the spectral rates, obtained with the spectral PEP formulation from Proposition \ref{thm:dig_conv} with the theoretical ones from \cite[Theorem 3.14]{DIGing} for a wide range of values for the step-size $\alpha$.  The spectral rates are identical for any value of $N$ and present a first regime where they decrease as $\alpha$ increases until a certain threshold step-size $\alpha_t$. This decrease is numerically close to $1-2\mu\alpha$. After $\alpha_t$, we observe a sharp increase in the spectral rates.
Since the theoretical results \cite{DIGing} depends on the value of $N$, Fig. \ref{fig:DIGing_alpha_evol} shows different curves, corresponding to theoretical bounds on the convergence rates for different numbers of agents $N$. All these curves also present two regimes, however, the decrease of the first regime is slower and the sharp increase takes place after a much lower threshold step-size. Moreover, both the sharp increase and the threshold step-size worsen as $N$ increases. \\
Therefore, the spectral rates obtained  with PEP allow for significant improvements in the tuning of the step-size $\alpha$ by choosing a larger value ($\approx \alpha_t$), which is independent of $N$. This leads to better convergence guarantees and therefore to better use of DIGing.
For example, in the setting of Fig. \ref{fig:DIGing_alpha_evol}, when $N=20$, the theoretical bound requires a step-size below $10^{-4}$, while the optimal step-size according to our spectral rate is around $4\times 10^{-3}$. This choice of the step-size allows improving the convergence guarantee by at least 2 orders of magnitude. \\
We observed the same two regimes for all the other values of $\mu$ and $\lambda$ we tested. However, the smaller the value of $\lambda$ is, the larger the gap between the values of the theoretical and spectral convergence rate at the threshold step-size. \smallskip

\paragraph{Impact of time-varying averaging matrices} \label{par:var_comp}
In the spectral formulation, we can choose to link different iterations together and to analyze the worst-case when we use the same constant averaging matrix $W$ at each iteration. We can also choose to consider each consensus step independently with potentially time-varying averaging matrices. For DIGing, we have chosen the second option to allow time-varying averaging matrices and be in the same condition as
\cite{DIGing}. However, we observe that both choices lead to the same worst-case values, even though the solution achieving these worst-cases may be different. One worst-case solution obtained with a constant matrix is therefore also a worst-case solution of the situation with time-varying matrices. We can thus analyze what is the worst constant matrix for DIGing and it will also be valid for time-varying settings. \smallskip

\paragraph{Worst averaging matrix and tightness analysis}
When the step-size $\alpha$ is optimized, i.e., equal to the threshold value ($\alpha_t$), we observe that the worst matrix for DIGing is the same as for DGD, and is thus given by $W^{(1)}$ in \eqref{eq:mat}. This matrix is only determined by the value of $\lambda$ and $N$ and the remainder it produces $\|\Yr - W^{(1)}\Xr\|$ is always close to zero in that case. This matrix $W^{(1)}$ is symmetric and generalized doubly stochastic. The same worst matrix is also recovered for larger step-size $\alpha$.
The bounds obtained using the exact PEP formulation with this specific matrix $W^{(1)}$ \emph{exactly} match the corresponding spectral bounds, within numerical errors. This means that the spectral formulation, even though it is a relaxation, provides again a \emph{tight performance bound} for DIGing with symmetric generalized doubly stochastic matrices and sufficiently large step-size. \smallskip

In summary, our spectral PEP formulation provides numerically tight convergence rates for DIGing that are independent of the number of agents $N$, and allows for better tuning of the constant step-size $\alpha$, leading to more efficient use of the DIGing algorithm.

\subsection{Accelerated Distributed Nesterov Gradient Descent}
As third use case, we analyze the accelerated distributed Nesterov
gradient descent (Acc-DNGD) algorithm proposed in \cite{AccDNGD}. We focus on the version designed for convex (not necessarily strongly convex) and $L$-smooth functions, described in Algorithm \ref{algo:ADNGD}.
\begin{algorithm}[b]
  \caption{Acc-DNGD}
  \begin{algorithmic}
  \STATE Initialize $x_i^0 = v_i^0 = y_i^0 = 0$ and $s_i^0 = \nabla f(0)$ for all $i$;
  \FOR{$k = 0, 1,\dots$}
  \FOR{$i = 1,\dots,N$} \vspace{1mm}
    \STATE  $x_i^{k+1} = \sum_{j=1}^N w_{ij} y_j^k - \etaADNGD_k s_i^k$; \vspace{1mm}
    \STATE  $v_i^{k+1} = \sum_{j=1}^N w_{ij} v_j^k - \frac{\etaADNGD_k}{\alphaADNGD_k} s_i^k$; \vspace{1mm}
    \STATE  $y_i^{k+1} = \alphaADNGD_{k+1} x_i^{k+1} + (1-\alphaADNGD_{k+1})v_i^{k+1}$; \vspace{1mm}
    \STATE $s_i^{k+1} = \sum_{j=1}^N w_{ij} s_j^k + \nabla f_i(y_i^{k+1}) - \nabla f_i(y_i^k)$; \vspace{1mm}
  \ENDFOR
  \ENDFOR
  \end{algorithmic}
  \label{algo:ADNGD}
\end{algorithm}
It achieves one of the best proved convergence rate in such setting, $\bigO\qty(\frac{1}{K^{1.4-\epsilon}})$ for any $\epsilon \in \qty(0,1.4)$, but there remains several open questions on choice of parameters and actual performance. We show how our technique allows shedding light on these questions.
We use the notations of \cite{AccDNGD}, which are slightly different from the rest of this paper. Here, $\eta_k$ denotes the diminishing step-size and $\alpha_k$ denotes a weighting factor.
Each agent $i$ keeps variables $x_i$, $v_i$, $y_i$, and $s_i$. The variables $s_i$ are local gradient tracking variables allowing each agent to estimate the average gradient $\frac{1}{N} \sum_{i=1}^N \nabla f_i(y_i)$.
The step-size $\etaADNGD_k$ is diminishing as
\begin{equation} \label{eq:accDNGD-sz}
  \etaADNGD_k = \frac{\etaADNGD}{(k+k_0)^\beta},
\end{equation}
where $\etaADNGD \in \qty(0,\frac{1}{L})$, $\beta \in  \qty(0,2)$ and $k_0 \ge 1$.
The sequence of $\alphaADNGD_k$ starts with
$\alphaADNGD_0 = \sqrt{\etaADNGD_0 L} \in \qty(0,1)$ and the next element of the sequence is each time computed as the unique solution in $(0,1)$ of
$$\alphaADNGD_{k+1}^2 = \frac{\etaADNGD_{k+1}}{\etaADNGD_k}(1-\alphaADNGD_{k+1})\alphaADNGD_k^2.$$
The convergence result \cite[Theorem 4]{AccDNGD} guarantees that the algorithm achieves an average functional error bounded as
$$f(\xb^k) - f(x^*) \le \bigO(\frac{1}{k^{2-\beta}}) \qquad \text{ for $\beta \in (0.6,2)$}.$$ Recall that $f(x) = \frac{1}{N} \sum_{i=1}^N f_i(x)$, $\xb^k = \frac{1}{N} \sum_{i=1}^N x_i^k$
and $x^*$ is a minimizer of $f$.
This convergence guarantee for Acc-DNGD only holds under specific conditions concerning the values of $\etaADNGD$ and $k_0$ \cite[Theorem 4]{AccDNGD}. These assumptions seem strong since they impose, in particular, that $\etaADNGD$ tends to 0 and $k_0$ tends to $\infty$ both when $\lam$ tends to 1 (disconnected graph) and to 0 (fully connected graph).
The authors of \cite{AccDNGD} conjecture that
\begin{enumerate}[(i)]
  \item Exact values of parameters $\etaADNGD$, $k_0$ do not actually matter and we can choose values that are not satisfying assumptions from \cite[Theorem 4]{AccDNGD}, and still obtain a rate of $\bigO(\frac{1}{k^{2-\beta}})$. For example, authors of \cite{AccDNGD} used $\etaADNGD = \frac{1}{2L}$, $k_0 = 1$ in their numerical experiments with $\beta = 0.61$.
  \item Choosing $\beta \in \qty[0,0.6]$ leads to the same rate $\bigO(\frac{1}{k^{2-\beta}})$. The case $\beta = 0$ uses constant step-size $\eta$ and is important because it would lead to a rate $\bigO(\frac{1}{k^{2}})$, i.e. the best possible rate in centralized framework for similar settings \cite[Theorem 2.1.7]{Nesterov}.
\end{enumerate}

\begin{figure*}[h]
\centering \hspace{-3mm}
    \begin{subfigure}{0.5\textwidth}
        \includegraphics[width=\textwidth]{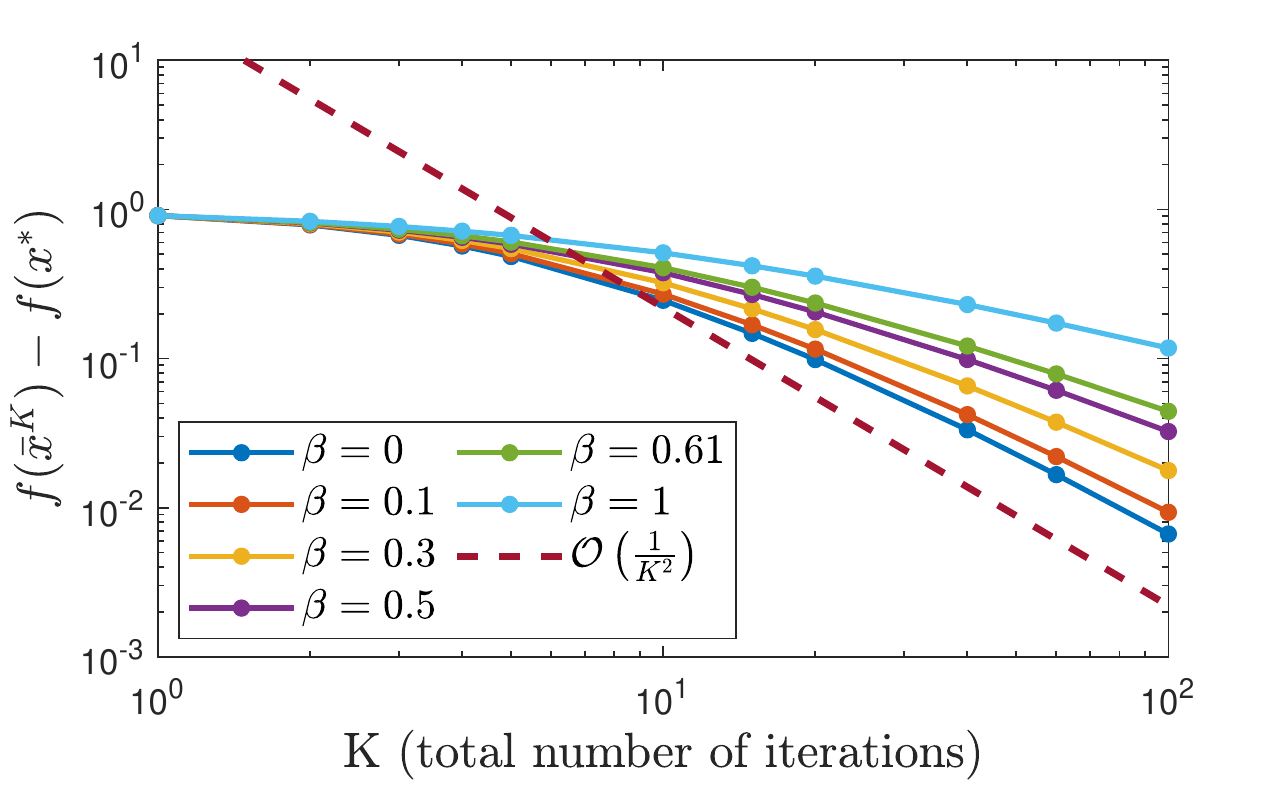}
        \caption{$\etaADNGD = 0.05$, $\lam = 0$}
        \label{fig:ADNGD_05_0}
    \end{subfigure} \hspace{-3mm}
    \begin{subfigure}{0.5\textwidth}
      \includegraphics[width=\textwidth]{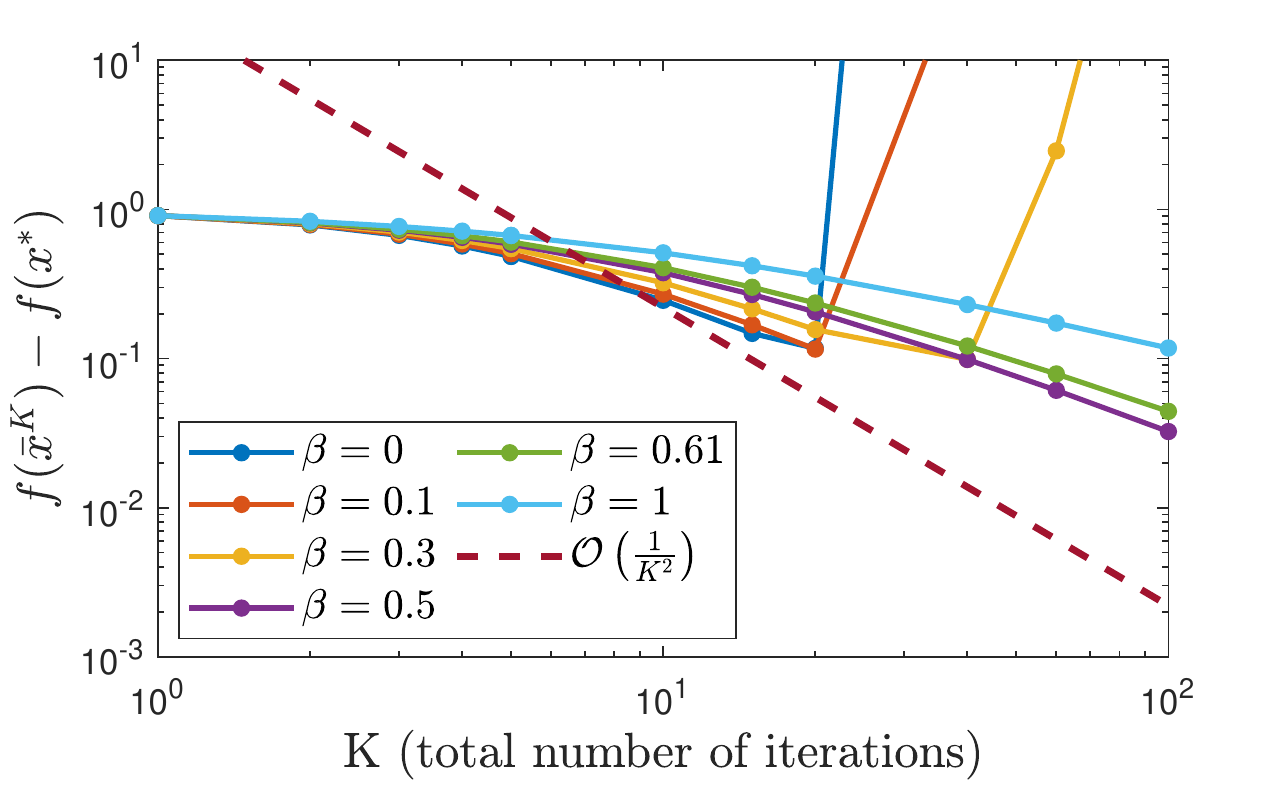}
        \caption{$\etaADNGD = 0.05$, $\lam = 0.75$}
        \label{fig:ADNGD_05_75}
    \end{subfigure}
    \\ \hspace{-3mm}
    \begin{subfigure}{0.5\textwidth}
        \includegraphics[width=\textwidth]{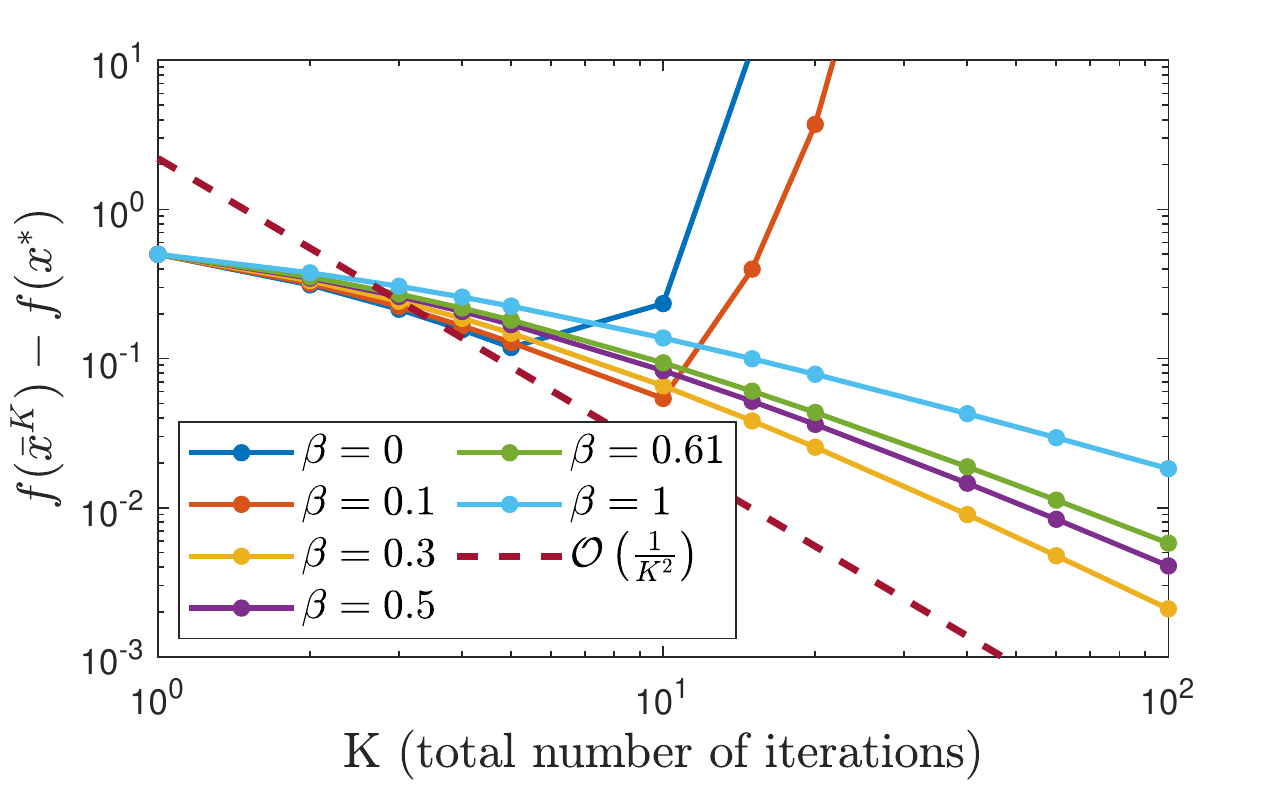}
        \caption{$\etaADNGD = 0.5$, $\lam = 0$}
        \label{fig:ADNGD_5_0}
    \end{subfigure} \hspace{-3mm}
    \begin{subfigure}{0.5\textwidth}
      \includegraphics[width=\textwidth]{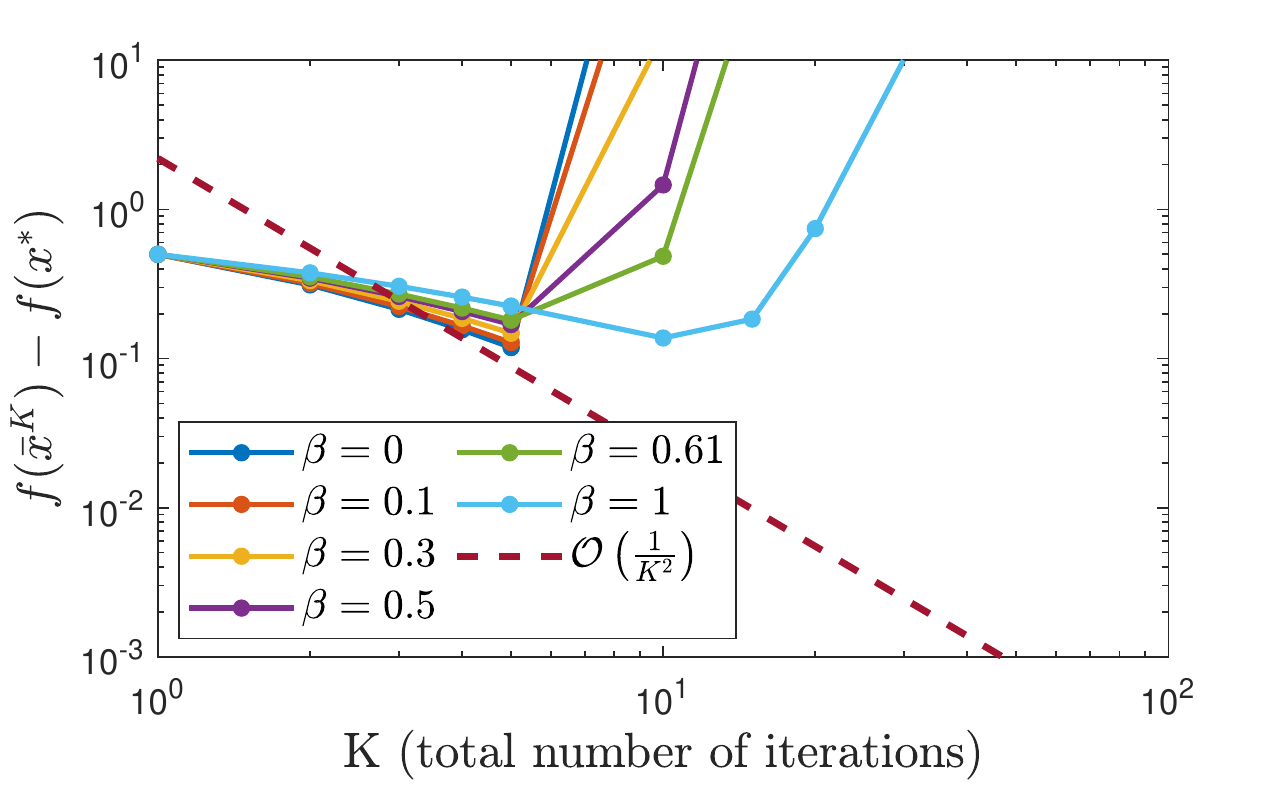}
        \caption{$\etaADNGD = 0.5$, $\lam = 0.75$}
        \label{fig:ADNGD_5_75}
    \end{subfigure}
\caption{Evolution with $K$ of the worst-case average functional error $f(\xb^K) - f(x^*)$ for Acc-DNGD, obtained with the spectral PEP formulation, with $N=2$, $ L=1$ and $k_0 = 1$. Different values of $\beta$, $\etaADNGD$ and $\lam$ are shown. Dashed lines show curves evolving in $\bigO(\frac{1}{K^2})$, which corresponds to the rate conjectured in \cite{AccDNGD} when $\beta = 0$. This asymptotic rate is reached for $\beta = 0$ only when $\etaADNGD$ is sufficiently small.}
\label{fig:ADNGD}
\end{figure*}

In this section, we use our spectral PEP formulation to analyze the Acc-DNGD algorithm, which allows having a better idea about the impact of the parameters on its performance and nuance the conjectures (i) and (ii) of \cite{AccDNGD}.
Fig. \ref{fig:ADNGD} shows the evolution of the worst-case average functional error $f(\xb^K) - f(x^*)$ with the total number of iterations $K$. We consider a symmetric range of eigenvalues $-\lm = \lp = \lam$, and thus we obtain spectral bounds valid over the entire class of matrices $\Wcl{-\lam}{\lam}$. The results are shown for different values of $\beta$, $\etaADNGD$ and $\lam$, while $L$, $N$ and $k_0$ are fixed.
We aim at testing the conjectures (i) and (ii).
To test (ii), about the validity range of $\beta$, conjectured to extend below 0.6, we choose $\beta = \{0,0.1,0.3,0.5\}$. To compare with proven valid values of $\beta$ and test conjecture (i), we also choose $\beta = \{0.61,1\}$.
We choose $\lam = \{0,0.75\}$ and $\eta = \{0.05, 0.5\}$ to have two representative values for each parameter (small and large).
We fix $L=1$ because the results for other values of $L$ can be recovered by scaling. We fix $N=2$ because the value of $N$ does not impact the worst-case value obtained with the spectral formulation, as for the two previous algorithms.
We fix $k_0 = 1$ for simplicity and because it does not affect the long term evolution of the performance. In particular, the value of $k_0$ has no impact when $\beta = 0$, see \eqref{eq:accDNGD-sz}.

We first observe that the value $\eta$ influences the performance of the algorithm. Indeed, when $\etaADNGD$ becomes larger (e.g., from Fig. \ref{fig:ADNGD_05_0}, \ref{fig:ADNGD_05_75} to Fig. \ref{fig:ADNGD_5_0}, \ref{fig:ADNGD_5_75}), the worst-case functional error decreases faster in a first phase but it sometimes explodes afterwards, preventing the algorithm to converge.
This occurs for example when $\beta$ is too small (e.g. $\beta = 0$ and $\beta = 0.1$ in Fig. \ref{fig:ADNGD_5_0}) or when $\lam$ is too large (e.g. in Fig. \ref{fig:ADNGD_5_75} and \ref{fig:ADNGD_05_75} where $\lam = 0.75$).
In Fig. \ref{fig:ADNGD_5_75}, this sharp increase even occurs in the experimental setting of \cite{AccDNGD}, i.e. $\etaADNGD = \frac{1}{2L} = 0.5$ and $\beta = 0.61$ (green line). These observations contradict conjecture (i) as currently stated.

Secondly, we observe in Fig. \ref{fig:ADNGD} that the curves for $\beta=0.5$ and $\beta=0.61$ behaves similarly in all the different plot. This suggests that the worst-case values present no phase transition at $\beta = 0.6$, and therefore, supports conjecture (ii) claiming that values of $\beta$ lower than 0.6 also provide rates $\bigO(\frac{1}{K^{2-\beta}})$, including the case $\beta = 0$.
The curves for $\beta = 0$ appear indeed to approach a decrease in $\bigO(\frac{1}{K^2})$, as conjectured in \cite{AccDNGD}, but only in Fig. \ref{fig:ADNGD_05_0} where the values $\etaADNGD$ and $\lam$ are sufficiently small.
Indeed, the rate $\bigO(\frac{1}{K^2})$ cannot be observed for larger values of $\etaADNGD$ or $\lam$ where PEP problems may become unbounded (see Figures \ref{fig:ADNGD_05_75}, \ref{fig:ADNGD_5_0} and \ref{fig:ADNGD_5_75}).
Therefore, the $\etaADNGD$ parameter should be tuned according to the values of $\lam$ and the choice of $\beta$. Qualitatively, we can see that it must decrease when $\lam$ increases or when $\beta$ decreases. Further analysis should be performed to tune the value of $\etaADNGD$ in general. Having sufficiently small step-sizes $\eta$ seems thus to be the key for conjecture (ii) to be true. Since smaller values for $\beta$ appear to require a smaller step-size, the performance would approach the convergence rate $\bigO(\frac{1}{K^{2-\beta}})$ more slowly in that case. However, even with well-tuned $\eta$, we cannot exclude that the worst-case performance of Acc-DNGD does not explode after a large number of iterations in some cases, as is the case in an early phase for some settings in Fig. \ref{fig:ADNGD}.

In addition, it is interesting to note that the worst matrix that is recovered from the spectral PEP formulation for Acc-DNGD is again $W^{(1)}$ from \eqref{eq:mat}, as it was the case for DGD and DIGing.
For Acc-DNGD also, the bounds obtained using the exact PEP formulation with this specific matrix $W^{(1)}$ \emph{exactly} match the corresponding spectral bounds, within numerical errors. This means that the spectral formulation, even though it is a relaxation, provides a \emph{tight performance bound} for Acc-DNGD with symmetric generalized doubly stochastic matrices.

\subsection{Code and Toolbox}
PEP problems have been written and solved using the PESTO Matlab toolbox \cite{PESTO}, with Mosek solver, within 200 seconds maximum. For example, for $N=3$, the time needed for a regular laptop to solve the spectral formulation for DGD (from Section \ref{sec:analysis_DGD}) is about 3, 12, 48, and 192 seconds respectively for $K=5, 10, 15, 20$.  The PESTO toolbox is available on \textsc{Github} (\url{https://github.com/PerformanceEstimation/Performance-Estimation-Toolbox}).
We have updated PESTO to allow easy and intuitive PEP formulation of gradient-based decentralized optimization methods, and we have added a code example for DGD, DIGing, and Acc-DNGD.

\section{Conclusion}
We have developed a methodology that automatically computes numerical worst-case performance bounds for any decentralized optimization method that combines first-order oracles with average consensus steps. This opens the way for computer-aided analysis of many other decentralized algorithms, which could lead to improvements in their performance guarantees and parameter tuning, and could allow rapid exploration of new algorithms. Moreover, the guarantees computed with our tool appear tight in many cases.

Our methodology is based on the performance estimation problem (PEP) for which we have developed two representations of average consensus steps.
Our first formulation provides the exact worst-case performance of the method for a specific network matrix.
The second formulation provides upper performance bounds that are valid for an entire spectral class of matrices. This spectral formulation often allows recovering the worst possible network matrix based on the PEP solutions.
We demonstrate the use of our automatic performance methodology on three algorithms and we discover, among other things, that DGD worst-case performance is much better than the theoretical guarantee; DIGing has performance independent of the number of agents and accepts much larger step-size than predicted by the theory; Acc-DNGD appears to present a rate $\bigO(\frac{1}{K^{2-\beta}})$, for $\beta \in \qty(0,2)$, only if the step-sizes are sufficiently small.

Further developments of our methodology may include a spectral formulation that is independent of the number of agents \cite{PEP_dec_Nindep} or that considers other classes of network matrices, e.g. non symmetric or B-connected networks.

\appendices
\section*{Acknowledgment}
The authors wish to thank Adrien Taylor for his helpful advice concerning the PESTO toolbox.

\section{Note on scaling of DGD}
\label{annexe:scaling}
In the DGD analysis, we have one constant $R$ to bound the subgradients
$\|\nabla f_i(x_i^k)\|^2 \le R^2$ (for $k = 0,\dots,K$), and another one $D$ to bound the initial distance to the optimum $\|x^0 - x^*\|^2 \le D^2$.
In our performance estimation problem, we consider general positive values for these parameters $D > 0$ and $R > 0$ and we parametrize the step-size by $\alpha = \frac{Dh}{R\sqrt{K}}$, for some $h > 0$. To pass from this general problem to the specific case where $D=1$ and $R=1$, that we actually solve, we consider the following changes of variables:
$$\tilde{x}_i = \frac{x_i}{D}, \quad \tilde{f}_i(\tilde{x}_i) = \frac{1}{DR} f_i(x_i) \quad \text{ and } \quad \tilde{\alpha} = \frac{R \alpha}{D}.$$
These changes of variables do not alter the updates of the algorithm or the nature of the problem. They allow expressing the worst-case guarantee obtained for $f(\xmoy ) - f(x^*)$ with general values of $D$, $R$, and $h$, denoted $w(D,R,h)$, in terms of the worst-case guarantee obtained for $\tilde{f}(\tilde{x}_{\mathrm{av}}) - \tilde{f}(\tilde{x}^*)$ with $D=R=1$, denoted $\tilde{w}(1,1,h)$:
\begin{equation} \label{eq:scal_wc}
  w(D,R,h) = DR~ \tilde{w}(1,1,h).
\end{equation}

The same kind of scaling can be applied to Theorem \ref{thm:bound}. The theorem is valid for general values of $D$ and $R$ but is specific to $\alpha = \frac{1}{\sqrt{K}}$, which is equivalent to picking $h = \frac{R}{D}$. After the scaling, we obtain the following bound, valid for $D = 1$, $R = 1$ and any value of $\alpha = \frac{h}{\sqrt{K}}$ with $h>0$:
\begin{equation} \label{eq:th_bound_scal}
  \tilde{f}(\tilde{x}_{\mathrm{av}} ) - \tilde{f}(\tilde{x}^*) \le \frac{h^{-1} + h}{2 \sqrt{K}} + \frac{2 h}{\sqrt{K}(1-\lam)}.
\end{equation}
This scaled theoretical bound with $h=1$ is equivalent to the bound from Theorem \ref{thm:bound} with $D=R=1$, which was the focus of the numerical analysis in Section \ref{sec:NumRes}. \\
This bound \eqref{eq:th_bound_scal} can be extended to any value of $D > 0$ and $R > 0$, using the relation from equation \eqref{eq:scal_wc}:
\begin{equation} \label{eq:th_bound_scal_2}
  f(\xmoy ) - f(x^*) \le DR \qty(\frac{h^{-1} + h}{2 \sqrt{K}} + \frac{2 h}{\sqrt{K}(1-\lam)}).
\end{equation}

\bibliographystyle{IEEEtran}
\bibliography{refs}
\end{document}